\documentclass[a4paper,11pt,pdftex]{article}

\usepackage{verbatim}
\usepackage{amsmath}
\usepackage{amssymb}
\usepackage{times}
\usepackage{amsthm}
\usepackage{xcolor}
\usepackage{enumerate}
\usepackage{enumitem}
\setlist[enumerate,1]{label={\normalfont(\roman*)}}
\usepackage{mathrsfs}
\usepackage[pdftex]{graphicx}
\usepackage{float}
\usepackage{wrapfig}
\usepackage{layout}
\usepackage{cite}
\usepackage{braket}
\usepackage[draft=false, setpagesize=false, pdfstartview=FitH,
colorlinks=true, linkcolor=blue, citecolor=magenta, pagebackref=true, bookmarks=false]
{hyperref}
 
\usepackage{lineno}
 \newcommand*\patchAmsMathEnvironmentForLineno[1]{
   \expandafter\let\csname old#1\expandafter\endcsname\csname #1\endcsname
   \expandafter\let\csname oldend#1\expandafter\endcsname\csname end#1\endcsname
   \renewenvironment{#1}
     {\linenomath\csname old#1\endcsname}
     {\csname oldend#1\endcsname\endlinenomath}}
 \newcommand*\patchBothAmsMathEnvironmentsForLineno[1]{
   \patchAmsMathEnvironmentForLineno{#1}
   \patchAmsMathEnvironmentForLineno{#1*}}
 \AtBeginDocument{
 \patchBothAmsMathEnvironmentsForLineno{equation}
 \patchBothAmsMathEnvironmentsForLineno{align}
 \patchBothAmsMathEnvironmentsForLineno{flalign}
 \patchBothAmsMathEnvironmentsForLineno{alignat}
 \patchBothAmsMathEnvironmentsForLineno{gather}
 \patchBothAmsMathEnvironmentsForLineno{multline}
 }

 \newcommand{\supp}{\mathrm{supp\,}}

\newcommand{\R}{\mathbb{R}}

\newcommand{\RN}{\mathbb{R}^N}

\newcommand{\N}{\mathbb{N}}
\newcommand{\e}{\varepsilon}
\newcommand{\rd}{\mathrm{d}}

\def\dist{\mathrm{dist}}
\def\A#1{A(\boldsymbol x,{#1})}

\theoremstyle{definition}
\newtheorem{definition}{Definition}[section]
\theoremstyle{plain}
\newtheorem{theorem}[definition]{Theorem}
\newtheorem{corollary}[definition]{Corollary}
\newtheorem{lemma}[definition]{Lemma}
\newtheorem{proposition}[definition]{Proposition}
\newtheorem{remark}[definition]{Remark}
\newtheorem*{notation}{Notation}

\usepackage[truedimen,margin=15truemm, bottom=20truemm]{geometry}

\numberwithin{equation}{section}

\usepackage{authblk}
 
\title
	{\bf Multi-bump solutions for sublinear elliptic equations with nonsymmetric coefficients}

  \author[a]{Chengxiang Zhang\thanks{  zcx@bnu.edu.cn}}
  \author[b]{Xu Zhang\thanks{ zhangxu0725@csu.edu.cn, darkblue1121@163.com, corresponding author  }}

  \affil[a]{\footnotesize Laboratory of Mathematics and Complex Systems (Ministry of Education), School of Mathematical Sciences, 
  
  Beijing Normal University, Beijing 100875, P. R. China}
  \affil[b]{\footnotesize  School of Mathematics and Statistics, HNP-LAMA, Central South University, Changsha 410083, P. R. China}
  
  \date{}
  
\begin{document}
  \maketitle
  \begin{minipage}{16.5cm}
    {\small {\bf Abstract:}
We investigate the existence of nonnegative bump solutions to the sublinear elliptic equation 
\[
\begin{cases}
-\Delta v - K(x)v + |v|^{q-2}v = 0 & \text{in } \mathbb{R}^N, \\
v(x) \to 0 & \text{as } |x| \to \infty,
\end{cases}
\]  
where \( q \in (1,2) \), \( N \geq 2 \), and the potential \( K \in L^p_{\mathrm{loc}}(\mathbb{R}^N) \) with \( p > N/2 \) is a function without any symmetry assumptions.  
Under the condition that \( \|K - 1\|_{L^p_{\mathrm{loc}}} \) is sufficiently small, we construct infinitely many solutions with  arbitrarily many  bumps.
The construction is challenged by the sensitive interaction between bumps, whose limiting profiles have compact support. The key to ensuring their effective separation lies in obtaining sharp estimates of the support sets. Our method, based on a truncated functional space, provides precisely such control. We derive qualitative local stability estimates in region-wise maximum norms that govern the size of each bump's essential support, confining its core to a designated region and minimizing overlap. Crucially, these estimates are uniform in the number of bumps, which is the pivotal step in establishing the existence of solutions with infinitely many bumps.

    
    }
    
    \end{minipage}

\section{Introduction and main Results}
We study the existence of nonnegative bump solutions of the following
sublinear Schr\"odinger equation
\begin{equation}\label{1.1} 
\begin{cases}
-\Delta v - K(x)v + |v|^{q-2}v = 0 & \text{in } \mathbb{R}^N, \\
v(x) \to 0 & \text{as } |x| \to \infty,
\end{cases}
\end{equation}
where $q \in (1,2)$, $N \geq 2$, and $K$ is a nonnegative function with no assumed symmetry. A distinctive feature of such sublinear elliptic
equations is that their nontrivial solutions often exhibit compact support.
This phenomenon is closely related to the study of blow-up sets in the porous
medium equation with a source term:
\begin{equation}\label{porousk}
u_t = \Delta u^m + \bar{K}(x) u^m,
\end{equation}
with $m = 1/(q-1) > 1$, where $\bar{K}(x) u^m$ acts as a spatially heterogeneous
source term. When $K \equiv 1$, this reduces to the homogeneous equation
\begin{equation}\label{porous}
u_t = \Delta u^m + u^m,
\end{equation}
whose solutions, for nonnegative compactly supported initial data, undergo
regional blow-up.
In \cite{Cortazar1998,Cortazar2002}, Cort\'azar, del Pino, and Elgueta proved
that for such initial data, the solution $u(x,t)$ of \eqref{porous} satisfies
\[
(\bar{T} - t)^{\frac{1}{m-1}} u(x,t) \to \sum_{i=1}^k w_*(|x - x_i|)
\quad \text{uniformly as } t \to \bar{T},
\]
where $\bar{T}$ is the blow-up time, $w_*$ is the unique nonnegative,
nontrivial radial solution of
\(
\Delta w_*^m + w_*^m = \frac{1}{m-1} w_*
\)
with support in a ball of radius $R^*$, and the points $x_1,\cdots,x_k \in
\mathbb{R}^N$ satisfy $\min_{i \neq j} |x_i - x_j| \geq 2R^*$.
By the change of variable
\(
w = (m-1)^{\frac{m}{m-1}} w_*^m,
\)
the profile $w$ satisfies the elliptic equation
\begin{equation}\label{1.1infty}
\begin{cases}
-\Delta v -  v + v^{q-1} = 0 & \text{in } \mathbb{R}^N, \\
v(x) \to 0 & \text{as } |x| \to \infty,
\end{cases}
\end{equation}
which coincides with \eqref{1.1} for $K \equiv 1$.
The basic properties of the nonnegative ground state\footnote[1]{ As in \cite{Zhang2025}, a ground state solution $ u $ refers to a solution that does not change sign, vanishes at infinity, and is such that the set $ \{x \mid u(x)\neq 0\} $ is connected.}  to \eqref{1.1infty}, including the uniqueness,  radial symmetry, and the structure of its
support, have been established in
\cite{Cortazar2002, ikomaNonlinearEllipticEquations2022,
guiSymmetryBlowsetPorous1995, serrinSymmetryGroundStates1999,
serrinUniquenessGroundStates2000}.
We summarize some of the  properties that will be used in this paper as follows:
\begin{lemma}\label{lemmaR*}
  Let $w$ be the nonnegative ground state of \eqref{1.1infty} with $w(0)=\max_{x\in \R^N} w(x)$. Then
 \begin{enumerate}
  \item $w\in C^2(\R^N)$, and there is $R^*>0$ such that $\supp w = \overline{B(0, R^*)}$.
  \item $w(x)=w(|x|)$ is strictly decreasing in $r=|x|$ for $r\in [0, R^*]$.
  \item  $w(R^*)=w'(R^*)=w''(R^*)=0$. Moreover,  $w(x) (R^*-|x|)^{-\frac{2}{2-q}}\to C_q$, as $|x|\to R^*$, where \[C_q :=\left(\frac{(2-q)^2}{2q}\right)^\frac{1}{2-q}.\]
  \item $w$ is the unique nonnegative ground state solution to 
  \eqref{1.1infty} up to translations in $\R^N$, and it is
  {\it nondegenerate} in the sense that if there exists
  $\psi\in H^1(B(0, R^*))$ satisfying $\int_{B(0, R^*)} w^{q-2}\psi^2<+\infty$
and 
\[
-\Delta \psi -\psi + (q-1) w^{q-2} \psi=0 \mbox{ in } B(0,R^*),
\]
then $\psi=\nabla w\cdot \lambda$  for some   $\lambda\in\R^N$ in $B(0,R^*)$ .
\item $w$ is the unique positive solution to the boundary value problem
\[
\begin{cases}
-\Delta u - u + u^{q-1} = 0 & \text{in } B(0, R^*), \\
u = 0 & \text{on } \partial B(0, R^*).
\end{cases}
\]
 \end{enumerate}
\end{lemma}
Analogously, for a solution   $u(x,t)$ of \eqref{porousk}, if the rescaled quantity
$(\bar{T} - t)^{\frac{1}{m-1}} u(x,t)$ converges as $t \to \bar{T}$, its limit
must be a solution of \eqref{1.1} (up to the above scaling transformation).
Consequently, the compact support of the self-similar blow-up profiles for the
parabolic problem \eqref{porousk} corresponds precisely to the compactly
supported solutions of the elliptic equation~\eqref{1.1}. This correspondence motivates the construction
of multibump solutions to \eqref{1.1} as a first step toward understanding
general blow-up sets.

Equation \eqref{1.1} serves as the sublinear analogue of the superlinear scalar
field equation
\begin{equation}\label{super}
-\Delta u + a(x)u = |u|^{p-2}u, \quad u \in H^1(\mathbb{R}^N),
\end{equation}
with a positive potential $a$. The existence and multiplicity of solutions to \eqref{super} have been
extensively studied since the pioneering work of Berestycki and Lions
\cite{Berestycki1983, Berestycki1983a}.
Under periodicity assumptions, Coti Zelati and Rabinowitz
\cite{cotizelatiHomoclinicTypeSolutions1992} constructed infinitely many
positive multibump solutions.
In \cite{Wei2010}, assuming that $a$ is radially symmetric and decays to a positive constant from above at a polynomial rate, Wei and Yan established the
existence of infinitely many multibump solutions via the Lyapunov--Schmidt
reduction method.
Cerami, Passaseo, and Solimini \cite{ceramiInfinitelyManyPositive2013}
further removed symmetry requirements, proving the existence of infinitely
many multibump solutions under mild slow-decay conditions on $a$; they later
constructed a solution with infinitely many bumps \cite{Cerami2015}.
Subsequently, Ao and Wei \cite{aoInfinitelyManyPositive2014} provided an
alternative proof of the result in \cite{ceramiInfinitelyManyPositive2013}
via a refined Lyapunov--Schmidt reduction.

Motivated by the correspondence with blow-up profiles and the lack of
existing multiplicity results for non-symmetric sublinear problems, we aim to
construct infinitely many nonnegative multibump solutions to \eqref{1.1}.
To state our main results, we impose the following assumptions on $K$:
\begin{enumerate}
  \item[(K1)] $K\in L^{p}_{loc}(\R^N)$, for some $p >N/2$.
  \item[(K2)]  $K(x)\leq a_1$ for some $a_1>0$.
  \item[(K3)] $K(x)\to  1$ as $|x|\to\infty$.
  \item[(K4)]  the set $ \set{x\in\R^N| K(x)\geq 1}$ is bounded.
\end{enumerate}
Our main results are as follows.
\begin{theorem}\label{them1.1}
  Assume (K1)--(K4). There exists a positive constant  $\alpha_0=\alpha_0(a_1 )>0$, such that \eqref{1.1} has infinity many solutions when $\|K-1\|_{p, loc}:=\sup_{x\in\RN}\|K-1\|_{L^p(B(x,1))}\leq\alpha_0$.
\end{theorem}
The proof of Theorem \ref{them1.1} also highlights the asymptotic properties of the solutions when $
\|K-1\|_{p, loc} \to 0.
$

\begin{theorem}\label{thm:1.3}
Let $(K_n)_n$ be any sequence of functions satisfying (K1)--(K4) (with the same values of $a_1$) and such that
\[
\|K_n-1\|_{p, loc} \to 0.
\]Then there exist $\bar{n} \in \mathbb{N}$ and a sequence $R_n > R^*$ with $R_n \to R^*$
as $n \to \infty$ such that for every $n > \bar{n}$ and every $k \in \mathbb{N} \setminus \{0\}$,
there exist $k$ points $x_1^n, \cdots, x_k^n \in \mathbb{R}^N$ and a nonnegative solution
$u_k^n$ of \eqref{1.1} with potential $K_n$ satisfying
\[
\supp u_k^n \subset \bigcup_{i=1}^k B(x_i^n, R_n),
\quad 
\lim_{n \to \infty} \max_{1 \le i \le k}
\bigl\| u_k^n(\,\cdot\, + x_i^n) - w \bigr\|_{L^\infty(B(0,R_n))} = 0,
\]
where $w$ is the unique nonnegative radial ground state solution to
\eqref{1.1infty}
and $R^* > 0$ denotes the radius of the support of $w$ (i.e., $\supp w = \overline{B(0,R^*)}$).
Moreover,
\[
\lim_{n\to\infty} \min\bigl\{ |x_i^n - x_j^n| : 1\leq i < j\leq  k \bigr\} \geq 2R^*.
\]
\end{theorem}
We can also obtain a solution with infinitely many bumps.
 \begin{theorem}\label{them2}
  Assume (K1)--(K4) and $\|K-1\|_{p, loc}\leq \alpha_0$. Then eqation \eqref{1.1} has a solution in $H_{loc}^1(\RN)$, which has infinity bumps with the following asymptotic properties:
  Let $(K_n)_n$ be any sequence of functions satisfying (K1)--(K4)  such that
\[
|K_n-1|_{p, loc} \to 0.
\]Then there exist $\bar{n} \in \mathbb{N}$ and a sequence $R_n > R^*$ with $R_n \to R^*$
as $n \to \infty$ such that for every $n > \bar{n}$  there exist $\{x_i^n\}_{i=1}^\infty$     and a nonnegative solution $u^n\in H^1_{loc}(\R^N)$ to  \eqref{1.1} with potential $K_n$ satisfying
\[
\supp u^n \subset \bigcup_{i=1}^\infty B(x_i^n, R_n),
\quad 
\lim_{n \to \infty} \sup_{i\geq 1}
\bigl\| u^n(\,\cdot\, + x_i^n) - w \bigr\|_{L^\infty(B(0,R_n))} = 0,
\quad 
\lim_{n\to\infty} \inf_{i\neq j} |x_i^n - x_j^n|  \geq 2R^*.
\]
\end{theorem}
In the superlinear setting, multibump solutions are typically constructed by
placing well-separated copies of a ground state far apart in space.  This is
feasible because bumps can be positioned arbitrarily far from one another,
and their mutual interaction decays exponentially fast.  Both features make
variational gluing arguments and the Lyapunov--Schmidt reduction method
effective.

In contrast, for the sublinear problem \eqref{1.1}, every nontrivial solution
has compact support (see Lemma~\ref{lemmaR*}).  This prevents spatial
separation of bumps and renders standard variational gluing techniques
inapplicable.  Moreover, the sublinear nonlinearity $|v|^{q-2}v$ with
$q \in (1,2)$ leads to a singular linearization: the derivative of the
nonlinearity behaves like $|v|^{q-2}$, which blows up in support set of $v$.
Consequently, the linearized operator around a compactly supported solution is
not well-defined on  a natural energy space, and the
Lyapunov--Schmidt reduction—which relies on a well-behaved invertible
linearized operator—cannot be applied.
To overcome these obstacles, we adopt the variational framework introduced in
\cite{ceramiInfinitelyManyPositive2013}.  However, a key difference arises:
because the limiting profile has compact support, a sequence of approximate
solutions cannot converge locally to isolated bumps because the energy estiamtes
 depends  sensitively on deviations  the support set of the sequence and the limiting profiles.  In particular, one cannot simply take local limits in
disjoint regions and patch them together.
This forces us to work more precisely within the support structure of candidate
solutions.  Given a configuration of points $\{x_j\}_{j=1}^k$, we consider a
constrained minimization problem
for the energy functional $I(u)$
on a Nehari-type set in
$
H_0^1( \cup_{j=1}^k B(x_j, R^* + d) )$, $d>0$,
where the local barycenter of the function in each ball $B(x_j, R^* + d)$ is
fixed at $x_j$.  We then take the
maximum of these local minimization values over all admissible configurations
$\{x_j\}_{j=1}^k$.
For sufficiently small $d > 0$, $\|K-1\|_{p,loc}$ and configurations satisfying
\[
(2R^*-\min_{i\neq j} |x_i-x_j|)^+ \to 0,
\]
the minimizers develop bumps that are nearly  non-overlapping,
each closely resembling a translate of the ground state $w$.  This
localization justifies viewing these minimizers as genuine candidate solutions.

To extend this finite-domain construction to the whole space, we establish
qualitative local stability estimates in region-wise maximum $H^1$-norms
:
\[
\| u - \sum_{j=1}^k w_j \|_{\boldsymbol{x}, d}
\lesssim \| \nabla_{\mathcal H} I^\infty(u) - \nabla_{\mathcal H} I^\infty ( \sum_{j=1}^k w_j  )  \|_{\boldsymbol{x}, d}.
\]
Here $w_j=w(\cdot -x_j)$, 
$\|u\|_{\boldsymbol{x}, d}=\max_{1\leq i\leq k}\|u\|_{H^1(B(x_i, R^*+d))}$,
$I^\infty$ is the energy functional for the autonomous problem, and 
$ \nabla_{\mathcal H} I^\infty(u)$ denotes   the component of the gradient of $I^\infty(u)$ tangent to the local barycenter constraints. 
 Formally, our qualitative stability estimate  resembles the ``stability of
critical points'' results for the Sobolev inequality in
\cite{Figalli2020, Ciraolo2018, Deng2025}.   When $\min_{i\neq j}|x_i-x_j|\geq 2R^*$,
it can be viewed as a nonlinear, quantitative extension of the nondegeneracy property to the multi-bump setting.
For fixed $d > 0$, this estimate implies that as
$\|K - 1\|_{p, loc} \to 0$, the candidate solutions converge
 on each $B(x_j, R^*)$ to the corresponding ground state.  In
particular, their supports converge to
$\bigcup_{j=1}^k \overline{B(x_j, R^*)}$.
Crucially, because the estimates are expressed in region-wise maximum norms,
they control each localized part of a function independently.  As a result,
the bounds remain uniform with respect to the number $k$ of bumps.  This
uniformity is the key ingredient that allows us to construct infinitely many
solutions with arbitrarily many bumps.

The organization of the remainder of this paper is as follows. In Section 2, we introduce the limit equation \eqref{1.1infty} and its fundamental properties, define the relevant function spaces and energy functionals, and present the key concepts of the "emerging part" and "local barycenters," which lay the groundwork for the subsequent construction of multi-bump solutions. Section 3 is devoted to establishing qualitative local stability estimates in region-wise maximum norms, which serve as the pivotal tool for controlling the interaction between bumps and ensuring their effective separation. In Section 4, we define a constrained minimization problem on bounded domains to construct solutions with any finite number of bumps and analyze the support and energy properties of these solutions. Section 5 extends the locally constructed solutions to the whole space, proves that they become global solutions of the original equation for $K$ sufficiently close to 1, and completes the proof of the main theorems (Theorems \ref{them1.1} and \ref{thm:1.3}) by employing a proof by contradiction alongside the stability estimates. 
Finally, in Section 6, by selecting a sequence of solutions with an increasing number of bumps and applying a compactness argument, we construct a solution possessing infinitely many bumps, thereby concluding the proof of Theorem \ref{them2}.

\begin{notation}
   Throughout this paper,
$2^*=+\infty$ for $N=1,2$ and $2^*=\frac{2N}{N-2}$ for $N\geq 3$; $ L^p(\mathbb R^N) \ (1\leq p<+\infty)$ is the usual Lebesgue space with the norm $ \|u\| _p =(\int_{\mathbb R^N}|u|^p)^{1/p}$;
$ H^1(\mathbb R^N)$ denotes the Sobolev space with the norm $\|u\|_{H^1}^2=\int_{\mathbb R^N}(|\nabla u|^2+|u|^2);$
$H^{-1}(\Omega)$ denotes the dual space of $H_0^1(\Omega)$
with the norm $\|h\|_{H^{-1}}=\sup \set{\langle h,u\rangle | u\in H_0^1(\Omega),\ \|u\|_{H^1}=1}$, where $\Omega$ is a domain in $\R^N$;
$o_n(1)$ (resp. $o_\varepsilon(1)$)
will denote a generic infinitesimal as $n\rightarrow \infty$ (resp. $\varepsilon\rightarrow 0^+$);  $B (x,\rho)$ denotes an open ball centered at $x\in\mathbb R^N$ with radius
$\rho>0$.
$a^\pm=\max\{0,\pm a\}$ for $a\in\mathbb R$.
 Unless stated otherwise,   $C, C'$ and $c$ are general constants.
 
\end{notation}

\section{Preliminary results}
This section presents preliminary results and notation concerning the
equation~\eqref{1.1infty}.
For $q\in(1,2)$, we define
$$X_q:=H^1(\R^N)\cap L^q(\R^N)\ \ \text{with\ norm\ }
\|u\|:=\|\nabla u\|_{2}+\|u\|_q.$$
Let $w$ be the unique radial ground state solution of \eqref{1.1infty}, then it 
is a critical point of  
\[
I^\infty (u)=\frac12(\|\nabla u\|_2^2 -\|u\|_2^2) +\frac1q\|u\|_q^q,\quad   u\in X_q.
\]
Moreover, it  can be characterized 
as the minimizer for 
\[
\inf_{u\in \mathcal N} I^\infty (u):=m_0,\quad 
  \mathcal N: =\Set{u\in X_q\setminus\{0\} | (I^\infty) ' (u) u=0}.
  \]

We now introduce some constants.
By Lemma \ref{lemmaR*} (iii),  
  fix a sufficiently small constant $\sigma_0\in(0,1)$, and define 
\[
\delta:=  w(R^*-4\sigma_0)< \sigma_0^2,   \quad \rho:= R^*-3\sigma_0,
\]
ensuring that 
\begin{gather}\label{eqsigma}
 w\left(\frac{(R^*)^2}{R^*+\kappa^{\frac{2-q}{3}}}\right)>2\kappa \quad \text{ for }  \kappa\in(0,\delta],\\
  R_0:=R^*+ \sigma_0^\frac{2-q}{2}<\frac2{\sqrt{3}}\rho=\frac{2}{\sqrt{3}}R^*-\frac{6}{\sqrt3}\sigma_0,\label{R 0}
\end{gather}
and  
\begin{equation}\label{eq9}
  0<\delta<\min\Big\{ \left(\frac{2R^*}{(R^*+1)^2a_1}\right)^\frac{3}{2(2-q)},\  \left(\frac{t_*}{a_1}\right)^\frac{1}{2-q},\ 
   (\frac{q-1}{2})^{\frac{1}{2-q}}  
  \Big\},\mbox{ for some $t_*\in(0,2-q)$,}
  \end{equation}
  where $a_1$ is given in (K2).
By \eqref{R 0},
  given any  three distinct points $x_i\in\R^N$, $i=1,2,3$, with $|x_i-x_j|\geq 2\rho$ for $i\neq j$,  we have  
\begin{equation}\label{eq8}
\bigcap_{i=1}^3 \overline{B(x_i,R_0)}=\emptyset. 
\end{equation}
Consequently, for any set of real functions $\{u_i\}_{i\geq 1}$,  such that  $\supp u_i\subset\overline{B(x_i, R_0)}$ for each $i \geq 1$, where $\min_{i\neq j} |x_i-x_j|>2\rho$ the following identity holds pointwisely
\begin{equation}\label{mima}
\sum_{i\geq 1} u_i(x) =u_{i_x}(x)+u_{j_x}(x)\ \ \text{or}\ \ \sum_{i\geq 1} u_i(x) =u_{i_x}(x),\ \ \text{for some }i_x,j_x\in\{1,2,\cdots\}\ \text{and }i_x\neq j_x.
\end{equation}
 \begin{lemma} \label{R0}
   Suppose that $\kappa\in (0, \delta]$, $\{x_j\}_{j=1}^k \subset \R^N$ satisfying $\min_{i\neq j} |x_i-x_j|>2\rho$, $v \in X_q$ is a nonnegative function   satisfying
    $$-\Delta v-a_1v+v^{q-1}
           \leq 0, \quad v\leq \kappa,\quad \text{ in\ }\mathbb R^N\setminus \bigcup_{i= 1}^k B(x_i,R^*) 
           $$
           in the weak sense.
   Then $v=0$  in $\mathbb R^N\setminus \bigcup_{i= 1}^k {B(x_i, R^*+\kappa^\frac{2-q}{3})}$. Specically, if $\kappa=\delta$,
   then $\supp v\subset \bigcup_{i= 1}^k {B(x_i, R_0)}$.
   \end{lemma}
   \begin{proof}
    Set \[t_\kappa=\frac{R^*}{R^*+\kappa^\frac{2-q}{3}}.\]
  By  $\kappa\in(0,\delta]$, $\delta<\sigma_0^2<1$ and \eqref{eq9}, we have
 $$a_1\kappa^\frac{2(2-q)}{3}(R^*+\kappa^\frac{2-q}{3})^2
 \leq a_1\kappa^\frac{2(2-q)}{3}(R^*+1)^2<2R^*< 2R^*+\kappa^\frac{2-q}{3}.$$ 
 Therefore,
 $$a_1\kappa^{2-q}(R^*+\kappa^\frac{2-q}{3})^2
< (2R^*+\kappa^\frac{2-q}{3})\kappa^\frac{2-q}{3},$$ 
which implies    
  $$a_1\kappa^{2-q}\leq \frac{2R^*\kappa^\frac{2-q}{3}+\kappa^\frac{2(2-q)}{3}}{(R^*+\kappa^\frac{2-q}{3})^2}=1-t_{\kappa}^2.$$
  By this we have
   \[
   -\Delta v\leq -(1-a v^{2-q})v^{q-1} \leq -t_\kappa^2v^{q-1} \quad \text{in } \mathbb{R}^N \setminus \bigcup_{i=1}^k B(x_i, R^*).
   \]
   On the other hand, for each $i$, the function $w_i(x) := w(t_\kappa(x - x_i))$ satisfies
   \[
   -\Delta w_i = t_\kappa^2w_i -  t_\kappa^2w_i^{q-1} \geq -t_\kappa^2w_i^{q-1} \quad \text{in } \mathbb{R}^N.
   \]
   Moreover, by $\kappa\leq\delta<\sigma_0^2<1$,
   $$\supp w_i\subset \overline{B(x_i,R^*+\kappa^{\frac{2-q}{3}})}\subset B(x_i,R^*+\sigma_0^{\frac{2-q}{2}})=B(x_i,R_0).$$
   Define $\tilde{w} := \frac12 \sum_{i=1}^k w_i$. Given that $q \in (1, 2)$ and $x\in\R^N$, using \eqref{mima}, we derive 
   \[
     \tilde{w}^{q-1}(x) = \left(\frac12 w_{i_x}(x) +\frac12 w_{j_x}(x)\right)^{q-1}
    \geq \frac12 \left(w_{i_x}(x)^{q-1} +  w_{j_x}(x)^{q-1} \right)= \frac12 \sum_{i=1}^k w_i^{q-1}(x),
   \]
 or
  \[
     \tilde{w}^{q-1}(x) = \left(\frac12 w_{i_x}(x) \right)^{q-1}
    \geq \frac12  w_{i_x}(x)^{q-1} = \frac12 \sum_{i=1}^k w_i^{q-1}(x).
   \]
   Consequently,  
   \[
   -\Delta \tilde{w} \geq - \frac{t_\kappa^2}2 \sum_{i=1}^k w_i^{q-1} \geq -t_\kappa^2  \tilde{w}^{q-1} \quad \text{in } \mathbb{R}^N.
   \]
   Moreover, by the choice of $\delta$ in \eqref{eqsigma}, we have 
  
   \[
   \tilde{w} \geq  \frac12 w(t_\kappa R^*)> \kappa \quad \text{on } \partial \bigcup_{i=1}^k B(x_i, R^*).
   \]
  Testing
   \[
   -\Delta(v - \tilde{w}) \leq -t_\kappa^2(v^{q-1} - \tilde{w}^{q-1}) \quad \text{in } \mathbb{R}^N \setminus \bigcup_{i=1}^k B(x_i, R^*)
   \]
   against $(v - \tilde{w})^+$, we deduce that $v \leq \tilde{w}$ in $\mathbb{R}^N \setminus \bigcup_{i=1}^k B(x_i, R^*)$, which completes the proof.
   \end{proof}

For a given at most countable set \(\{x_j\}_{j\in \mathcal J} \subset \mathbb{R}^N\) satisfying the condition that \(|x_i - x_j| \geq 2\rho\) for all \(i \neq j\), we define the following index set for each \(x_i\):
\[
\mathscr{I}_i = \left\{ j \in \mathcal J \mid x_j \in B(x_i, 3R_0) \right\}.
\]
Here, \(\mathcal J\) can either be \(\{1, \cdots, k\}\) for some positive integer \(k\) or the set of all positive integers \(\mathbb{N} \setminus \{0\}\). Without loss of generality, we will denote a at most countable set \(\{x_j\}_{j\in \mathcal J}\) as \(\{x_j\}_{j \geq 1}\) in the subsequent discussion.
Since \eqref{R 0}: $R_0<\frac{2}{\sqrt 3}\rho$, 
we observe that the number of elements in $\mathscr{I}_i$ is bounded by 
\begin{equation}\label{eq 9}
\#\mathscr{I}_i\leq k_0:=\left\lfloor\left(\frac{4R_0}{\rho}\right)^N\right\rfloor+1
<\left(\frac{8}{\sqrt 3}\right)^N+1.
\end{equation}

\begin{lemma}\label{lem 2.3}
For any sequence $\{x_j\}_{j\geq 1}\subset \R^N$ satisfying $\inf_{i\neq j}|x_i - x_j| \geq 2R^* - \sigma$ with   $\sigma\in[0,\sigma_0]$, it holds
\[\left\lVert\bigg(\sum_{j\geq1}w(\cdot-x_j)\bigg)^{q-1}-\sum_{j\geq1}w(\cdot-x_j)^{q-1}\right\rVert_{L^r(B(x_i,2R_0))}
\leq C\sigma^{\frac{2(q-1)}{2-q}+\frac{N+1}{2r}},\]
where    $r\in [1,+\infty]$, $i\geq 1$ and $C$ is a positive constant depending only on $q, N$, and $k_0$.
\end{lemma}
\begin{proof} 
Since each $w(\cdot - x_j)$ is supported in $B(x_j, R^*)$, and  by \eqref{eq8}, the support of the function
$$h:=\Big(\sum_{j\geq1}w(\cdot-x_j)\Big)^{q-1}-\sum_{j\geq1}w(\cdot-x_j)^{q-1} $$
is the union of
\[
A_{j\ell}:= \overline{B(x_j,R^*)\cap B(x_\ell,R^*)}, \ \ j\neq \ell.
\]
By Lemma \ref{lemmaR*} (iii) and $|x_j-x_\ell|\geq 2R^*-\sigma$ for $j\neq \ell$, we conclude that $\sup_{A_{j\ell}} |h|\leq C_{q,N}  \sigma^\frac{2(q-1)}{2-q}$ and  the measure of each $A_{j\ell}$ is less than 
$C_{q,N}\sigma^\frac{N+1}2$ for some   constant $C_{q,N}>1$ depending only on $q, N$.
Therefore,
the conclusion holds for $r=+\infty$. For $r\in[1,+\infty)$
and each $j\neq \ell$,
\[ \int_{A_{j\ell}}|h|^r
\leq C_{q,N}^{r+1} \sigma^{\frac{2r(q-1)}{2-q}+\frac{N+1}{2}}. 
\]
By \eqref{eq 9}  
we know that $B(x_i,2R_0)$ intersects   at most  $k_0-1$ distinct sets $A_{j\ell}$. Therefore, we conclude the proof.
\end{proof}

 \begin{lemma}\label{lem2.2}
Let   $ \{x_j^n\}_{j\geq 1}$  be a sequence of point  sets  satisfying
$\liminf_{n\to+\infty}\inf_{i\neq j}|x_i^n-x_j^n| \geq 2R^*$, and let $u_n \in H_{loc}^1(\R^N)$
satisfying  $ u_{n}(\cdot + x_{i_n}^n) \rightharpoonup v$ weakly in \(H^1(B(0, R^* + d))\)   for some $d>0$ and $v\in H^1(B(0, R^*+d))$.
Assume that    $u_n =0$ in $B(x_{i_n}^n, R^*+d)\setminus \bigcup_{j\in \mathscr I_{i_n}^n} \overline{B(x^n_j, R_n)}$
where \(R_n > 0\), \(R_n \to R^*\) as \(n \to \infty\), and \(\mathscr{I}_{i_n}^n = \{ j \mid x_j^n \in B(x_{i_n}^n, 3R_0) \}\).  
Then
$v|_{B(0, R^*)} \in H_0^1(B(0, R^*))$.
\end{lemma}
\begin{proof}
   Assume without loss of generality that  $x_{i_n}^n=0$.
  To show  $v|_{B(0, R^*)} \in H_0^1(B(0, R^*))$, it suffices to show the trace 
  $Tv=0$,  where $T:H^1(B(0, R^*+d))\to H^{1/2}(\partial B(0, R^*))$
    is
    the trace operator.
    
  Set 
  $\widetilde u_n (x)= u_n(\frac{R_n}{R^*}x)$, and $\widetilde x_j^n = \frac{R^*}{R_n} x_j^n$, we know 
  $\widetilde u_n (x) \rightharpoonup v$ in $H^1(B(0, R^*+d/2))$ 
 and 
 $\widetilde u_n =0$ in $B(0, R^*+d/2)\setminus \cup_{j\in\mathscr I_{i_n}^n}  \overline{B(\widetilde x^n_j, R^*)}$.
Since $T$ is a bounded linear operator,
 $T\widetilde u_n \rightharpoonup Tv$ in $H^{1/2}(\partial B(0, R^*))$.
By \eqref{eq 9}, we have the fact that $\# \mathscr{I}_{i_n}^n\leq k_0$. So we can assume without loss of generality that 
$\{\lim_{n\to\infty} \widetilde x_j^n \mid j\in \mathscr I_{i_n}^n\}=\{x_0,x_1,\cdots,x_{k-1}\}$ with $x_0=0$ and $k\leq k_0$. Since
\[\liminf_{n\to+\infty}\inf_{i\neq j}|\widetilde x_i^n-\widetilde x_j^n|=\liminf_{n\to+\infty}\frac{R^*}{R_n}\inf_{i\neq j}|x_i^n-x_j^n| \geq 2R^*,\]
we have $|x_j|\geq 2R^*$, $j=1,\cdots,k-1$.

Let $\sigma>0$ be any fixed small constant. 
Set $S_\sigma=\partial B(0, R^*)\cap \left(\bigcup_{j=1}^{k-1} B(x_j, R^*+\sigma)\right)$. By $|x_j|\geq 2R^*$ for $j=1,\cdots,k-1$, the $N-1$ dimensional Hausdorff measure for $S_\sigma$ is bounded by $C\sigma^\frac{N-1}2$.
Moreover,  it holds
$\widetilde u_n=0$ in $B(0, R^*+d/2)\setminus \left(\overline{B(0,R^*)}\cup \bigcup_{j=1}^{k-1}B(x_j,R^*+\sigma/2)\right)$ for large $n$. And hence,  for large $n$,
\begin{equation}\label{eq 10}
\|T\widetilde u_n\|_{L^2(\partial B(0,R^*)\setminus S_\sigma)}\leq C(\sigma)
\|\widetilde u_n\|_{H^1\left(B(0, R^*+d/2)\setminus(\overline{ B(0,R^*)}\cup \bigcup_{j=1}^{k-1}B(x_j,R^*+\sigma/2))\right)}=0.
\end{equation}
On the other hand, there are positive constants $C, C'$ indendent of $\sigma$ such that, for all $n$,
\begin{equation}\label{eq 11}
\|T\widetilde u_n\|_{L^2(\partial B(0,R^*))}\leq C
\|\widetilde u_n\|_{H^1( B(0,R^*))}\leq C'.
\end{equation} 
 Then for any $\psi\in C^\infty(\partial B(0, R^*))$, we have by \eqref{eq 10} and \eqref{eq 11} that
 \[
 \int_{\partial B(0, R^*)} |T\widetilde u_n \psi| = \int_{S_\sigma} |T\widetilde u_n \psi |
 \leq  \|T\widetilde u_n\|_{L^2(\partial B(0, R^*))} \|\psi\|_{L^\infty(\partial B(0, R^*))}  |S_\sigma|^\frac12 \leq C'\|\psi\|_{L^\infty(\partial B(0, R^*))}\sigma^{\frac{N-1}4}.
 \]
Therefore, for any $\psi\in C^\infty(\partial B(0, R^*))$ \[\int_{\partial B(0, R^*)} Tv \psi =\lim_{n\to+\infty} \int_{\partial B(0, R^*)} T\widetilde u_n \psi \leq C'\|\psi\|_{L^\infty(\partial B(0, R^*))}\sigma^{\frac{N-1}4}.
\] 
Recalling that $\sigma>0$ could be any fixed constant and the constant $C'$ is indendent of $\sigma$, we get $\int_{\partial B(0, R^*)} Tv \psi=0$ and hence \(Tv = 0\) in \(L^2(\partial B(0, R^*))\). Therefore, \(v|_{B(0, R^*)} \in H_0^1(B(0, R^*))\).  
\end{proof}
Let $\delta$, $\sigma_0$ and $\rho$ be constants defined at the beginning of this section.
Adopting the notion from  \cite{ceramiInfinitelyManyPositive2013}, we define the concept of an emerging function as follows:

For all functions \(u \in H^1_{loc}(\R^N)\), we set
\[u^{\delta}(x):= {(u - \delta)}^{+}(x),\quad u_{\delta}(x):= u(x) - u^{\delta}(x),\]
and we call \(u^{\delta}\) the {\it emerging part} of \(u\) above \(\delta\) and \(u_{\delta}\) the  {\it submerged part} of \(u\) under \(\delta\).
We say $u\in H^1_{loc}(\R^N)$ is {\it emerging around $\{x_i\}_{i\geq 1}$} if
$$u^\delta=\sum_{i\geq 1} u_i^\delta,\ \ \text{where } u_i^\delta\geq (\not\equiv)0,\ u_i^\delta\in H_0^1(B(x_i,\rho)),\ B(x_i,\rho)\cap B(x_j,\rho)= \emptyset,\ i\neq j.$$
In the following proofs, the emergence of \(u\) around \(\{x_j\}_{j \geq 1}\) implies that the points \(\{x_j\}\) are pairwise separated, satisfying \( |x_i - x_j| \geq 2\rho = 2R^* - 6\sigma_0 \) for any \( i \neq j \).

Given a function $u \in H^1_{loc}(\mathbb{R}^N)$ that emerges around the points $\{x_i\}_{i\geq 1}$, 
we define the notion of the {\it local barycenters} of $u$ as
\[{\beta}_{i}(u) = \frac{1}{{\| u_{i}^{\delta} \|}_{2}^{2}}\int_{\R^N}( x {-} x_{i} ){( u_{i}^{\delta}(x) )}^{2}dx,\quad  i\geq 1.\]
We analyze candidate solutions with prescribed local barycenters, which introduces Lagrange multipliers into the framework. The following lemma establishes an initial estimate for these multipliers:
\begin{lemma}\label{lem2-3}
  Assume that $u \in H^1_{loc}(\mathbb{R}^N)$  emerges around the points $\{x_j\}_{j\geq 1}$, and for some $h\in H^{-1}(B(x_i, \rho))$ and $\lambda \in \R^N$, the equation
  \[
 \begin{aligned}
   -\Delta u -u +|u|^{q-2}u=\lambda\cdot (x-x_i-\beta_i(u)) u_{i}^{\delta}+h \end{aligned} 
  \] holds weakly in  $B(x_i, \rho)$.
Then there exist constants $\e_0\in(0,1)$ and $ C>0$ such that
\begin{equation}\label{lamb}
  |\lambda|\leq C\left(\|u(\cdot+x_i)-w \|_{L^2(B(0,\rho))} +\|h\|_{H^{-1}(B(x_i,\rho))}\right)
\end{equation}
provided that
$\|u(\cdot+x_i)-w\|_{L^2(B(0,\rho))}\leq \e_0$.
Furthermore, if 
$v \in H^1_{loc}(\mathbb{R}^N)$  emerges around the points $\{x_j\}_{j\geq 1}$, and for some $\bar h\in H^{-1}(B(x_i, \rho))$ and $\bar\lambda \in \R^N$, the equation
\[
   -\Delta v -v +|v|^{q-2}v=\bar \lambda\cdot (x-x_i-\beta_i(v)) v_{i}^{\delta}+\bar h
\] is satisfied,
then
\begin{equation}\label{lambdabar}
  |\lambda-\bar\lambda|\leq C\left( (1+|\bar\lambda|)\|u-v\|_{L^2(B(x_i,\rho))}   +\|h-\bar h\|_{H^{-1}(B(x_i,\rho))}\right)
\end{equation}
provided that $\|v(\cdot+x_i)-w\|_{L^2(B(0,\rho))}\leq \e_0$ and $\inf_{B(x_i,\rho)} v\geq \delta/2$. Here $\e_0>0$ and $ C>0$ are independent of $u$ and $v$.
\end{lemma}
\begin{proof}
 Without loss of generality we assume $x_i=0$. Then in $B(0,\rho)$, it holds
\begin{equation}\label{eq5}
\begin{aligned}
-\Delta (u-v)- (u-v)+|u|^{q-2}u-v^{q-1}
=& (\lambda-\bar\lambda)\cdot(x-\beta_i(u))u_i^\delta
 +\bar\lambda\cdot(\beta_i(v)-\beta_i(u))v_i^\delta\\
&+ \bar\lambda\cdot(x-\beta_i(u))(u_i^\delta-v_i^\delta) +h-\bar h.
\end{aligned}
\end{equation}
In $B(0,\rho)$, since $v \geq \delta/2$, we have the estimate
 \begin{equation}\label{a.e.}
   ||u|^{q-2}u-v^{q-1}|\leq v^{q-1}| |u/v|^{q-1}-1|\leq C v^{q-1}|u/v-1|\leq C\delta^{q-2}|u-v|.
 \end{equation}
Moreover, according to the definition of $\beta_i$, we have   
$|\beta_i(u)-\beta_i(v)| \leq C\|u-v\|_{L^2(B(0,\rho))}$, and $|\beta_i(u)|=|\beta_i(u)-\beta_i(w)| \leq C\e_0$.

Set  $\varphi=(\lambda-\bar \lambda) \cdot x \psi$, where $\psi\in C_0^\infty(B(0,\rho))\setminus\{0\}$ is a radial nonnegative function.
Test the equation \eqref{eq5} against $\varphi$ and integrate over $B(0,\rho)$. We   obtain by \eqref{a.e.} that
 the left-hand side is
 \[\int_{B(0,\rho)}(u-v)(-\Delta \varphi-\varphi) +(|u|^{q-2}u-v^{q-1}) \varphi=
 O(|\lambda-\bar\lambda|\|u-v\|_{L^2(B(0,\rho))}),
 \] 
 and by  $|u^\delta-v^\delta|\leq |u-v|$ that the right-hand side is
 $$
  \begin{aligned}
 \int_{B(0,\rho)}u_i^\delta\psi((\lambda-\bar \lambda)\cdot x)^2 +O(\e_0|\lambda-\bar\lambda|^2|)+ O\left( |\lambda-\bar\lambda|(|\bar\lambda|\|u-v\|_{L^2(B(0,\rho))}+\|h-\bar h\|_{H^{-1}(B(0,\rho))})\right).
  \end{aligned}
  $$
Since $|u^\delta-w^\delta|\leq |u-w|$, the leading term equals
\[
 \int_{B(0,\rho)}  w^\delta\psi((\lambda-\bar \lambda)\cdot x)^2  +O(\e_0|\lambda-\bar\lambda|^2) =\frac{|\lambda-\bar\lambda|^2}{N} 
 \int_{B(0,\rho)}w^\delta\psi  |x|^2 +O(\e_0|\lambda-\bar\lambda|^2).
 \]
Then we obtain \eqref{lambdabar} if $\e_0$ is fixed sufficiently small.
By setting $v = w$ in \eqref{eq5}, we obtain $\bar\lambda = 0$ and $\bar h = 0$; thus, \eqref{lamb} becomes a special case of \eqref{lambdabar}.
\end{proof}

\begin{lemma}\label{lem2-4}
  Assume that $u,v \in H^1_{loc}(\mathbb{R}^N)$  emerge around the points $\{x_j\}_{j\geq 1}$.
  Then 
\[  \left|\int_{B(x_i,\rho)}
 2(x-x_i)v^\delta(u- v)-
   \beta_i(u)\|u_i^\delta\|_{L^2}^2+\beta_i(v)\|v_i^\delta\|_{L^2}^2\right|\leq \rho \|u-v\|_{L^2(B(0,\rho))}^2.
 \]
\end{lemma}
 \begin{proof}

 Set
$\mathcal R=( u^\delta)^2-(v^\delta)^2-2v^\delta(u-v)$. By 
$|u^\delta-v^\delta|\leq |u-v|$, we have
$ 0 \leq   \mathcal R 
\leq (u-v)^2,
$ 
and hence 
\[
\int_{B(x_i, \rho)}|x-x_i|\mathcal R \leq \rho \|u-v\|_{L^2(B(0,\rho))}^2.
\]
Then the conclusion follows from  the fact
$2(x-x_i)v^\delta(u- v)-
\beta_i(u)\|u_i^\delta\|_{L^2}^2+\beta_i(v)\|v_i^\delta\|_{L^2}^2=-2(x-x_i)\mathcal R$.
 \end{proof}

\section{Local stability estimates}
In Section 3, we establish key \emph{qualitative local stability estimates} in region-wise maximum $H^1$-norms. These estimates provide precise control over the deviation of a candidate multi-bump function from a sum of translated ground states. Crucially, they quantify how this deviation is governed by the corresponding Euler--Lagrange residuals, and their uniformity with respect to the number of bumps is a pivotal feature. These stability results serve as the primary analytical tool in Section 5, where they are used to prove that the localized minimizers constructed in Section 4 are indeed solutions of equation \eqref{1.1} as the potential $K$ approaches 1.

\ \ \ \ \ \ \ \

For an  at most countable set $\{x_j\}=\{x_j\}_{j\geq 1}$, we set 
\begin{equation}\label{sigma}
   \sigma(\{x_j\})=(2R^*-\inf_{j\geq 1}|x_i-x_j|)^+ \mbox{ if } \#\{x_j\}\geq 2,
  \mbox{ and }
    \sigma(\{x_j\})=0\mbox{ if } \#\{x_j\}=1.
\end{equation}
We define the set for $k\in\N\setminus\{0\}$,
\[
  \mathcal K_k:=\Set{ \{x_j\}_{j=1}^k \subset \R^N | \sigma(\{x_j\}_{j=1}^k)\leq \sigma_0}.
\]
In our analysis,  we treat   elements in $ \mathcal K_k$ as   unordered sets,  independent of the indexing order.
Thus, if     $\boldsymbol{x}\in \mathcal K_{k_1}$ and $\boldsymbol{y}\in \mathcal K_{k_2}$ satisfy
\[
\inf\{|x_i-y_j| \ |\  x_i\in \boldsymbol{x}, y_j\in \boldsymbol{y} \}\geq 2R^*-\sigma_0,
\]
  the union
$\boldsymbol{x}\cup \boldsymbol{y}$ forms a well-defined element in $\mathcal K_{k_1+k_2}$.
When the topology in $\mathcal{K}_k$ is concerned, we use the metric 
\[
  \dist(\boldsymbol{x},\boldsymbol{y}):=\max_{x\in \boldsymbol{x}}\min_{y\in\boldsymbol{y}}|x-y|+
  \max_{y\in\boldsymbol{y}}\min_{x\in \boldsymbol{x}}|x-y|,
   \mbox{ for } \boldsymbol{x},\boldsymbol{y}\in \mathcal K_k.
\]

Let $k\in\N\setminus\{0\}$, $\boldsymbol x=\{x_j\}_{j=1}^k\in\mathcal K_k$ and
$d\geq 0$.  We use the notation $\sigma(\boldsymbol x)=\sigma(\{x_j\})$, where $\sigma(\{x_j\})$ is given in \eqref{sigma}, and we define
\[ A(\boldsymbol x,d):=  \bigcup_{j=1}^k B(x_j,R^*+d).\]
In $H_0^1(\A d)$, 
we define a  norm 
\[
\|u\|_{\boldsymbol x, d}=\max_{1\leq j\leq k}\|u\|_{H^1(B(x_j, R^*+d))},
\]
which is equivalent to the standard $H^1$ norm on  $H_0^1(\A d)$.
For $h\in H^{-1}(\A d)$, we define a norm
\[
\|h\|_{*, \boldsymbol x, d}=\max_{1\leq j\leq k} \|h\|_{H^{-1}(\A d\cap B(x_j, 2R_0))}.
\]
We also set
\[
H_{\boldsymbol x, d}:=\set{u\in H_0^1(\A d) |  u\geq 0 \text{ is emerging around $\boldsymbol x$},  \text{ and } \beta_j(u)=0, \ 1\leq j\leq k}.
\]
For nonnegative $u\in H_0^1(\A d)$ that is  emerging around $\boldsymbol x$, and $h \in H^{-1}(\A d)$, we define the region-wise projected semi-norm
\[
\begin{aligned}
\|h\|_{*,u}
&=\max_{1\leq j\leq k}\min_{\lambda_j\in\R^N}\|h- \lambda_j \beta_j'(u)\|_{H^{-1}(\A d\cap B(x_j, 2R_0))}\\
&=\max_{1\leq j\leq k}\min_{\lambda_j\in\R^N}\|h- \lambda_j\cdot (x-x_j)u^\delta_j\|_{H^{-1}(\A d\cap B(x_j, 2R_0))},
\end{aligned}
\]
where, by abuse of notation, the function $\lambda_j\cdot (x-x_j)u^\delta_j$ inside the norm is identified with the corresponding distribution via the $L^2$ inner product.
 Obviously,   
 \[ 
  \|h \|_{*,u}\leq
   \|h\|_{*,\boldsymbol x,d}.
 \]
And we have
\begin{equation}\label{K-1}
	\|(I^\infty)'(u)-I'(u)\|_{*,\boldsymbol x,d} 
	\leq C\|K-1\|_{p,loc}\|u\|_{x,d}.
\end{equation}
Moreover, by Lemma \ref{lem 2.3}, if $d\leq \sigma_0$,
\begin{equation}\label{Iinfw}
 \bigg\|(I^\infty)'\big(\sum_{j= 1}^k  w(\cdot-x_j)\big)\bigg\|_{*,\boldsymbol x,d} 
  \leq 
\max_{1\leq j\leq k}\left\lVert\bigg(\sum_{j\geq1}w(\cdot-x_j)\bigg)^{q-1}-\sum_{j\geq1}w(\cdot-x_j)^{q-1}\right\rVert_{L^r(B(x_j,2R_0))} \leq C \sigma(\boldsymbol x),
 \end{equation}
 where $r > 2N/(N+2)$ is such that $\frac{2(q-1)}{2-q}+\frac{N+1}{2r}\geq 1$.
\begin{proposition} \label{prop3.1}
There exist   $\varepsilon_1 \in(0,\e_0)$ and $\sigma_1\in(0,\sigma_0)$  and $C>0$ such that the following statement holds:  for any $d\in[0,\sigma_0]$,  any finite set $\boldsymbol x=\{x_j\}_{j= 1}^k\subset \R^N$ satisfying
  $\sigma(\boldsymbol x)\leq \sigma_1$, and any 
   $u\in H_{\boldsymbol x, d}$ satisfying
\[
\max_{1\leq j\leq k}\big\| u -   w(\cdot - x_j) \big\|_{L^2(B(x_j, R^*))} 
\leq \varepsilon_1,
\]
there holds
    \[  \big\|u-\sum_{j= 1}^k  w(\cdot-x_j)\big\|_{\boldsymbol x, d}\leq C \bigg\|(I^\infty)'(u)-(I^\infty)'\big(\sum_{j= 1}^k  w(\cdot-x_j)\big)\bigg\|_{*,u}.\]
\end{proposition}
\begin{remark}
	(i) Since $H_0^1(\A d)$ is Hilbert space 
	under the inner product
	\[
	\langle u, v\rangle =\int_{\A d} (\nabla u\nabla v + uv)
	\]
	there is an 
	isometry $\mathcal L$ for $H^{-1}(\A d)$ to $H_0^1(\A d)$.
	Since 
	\[
	\|h\|_{H^{-1}(\A d\cap B(x_i, 2R_0))}^2 \leq
	\|\mathcal L h\|_{H^1(\A d\cap B(x_i, 2R_0))}^2 \leq 
	\sum_{j\in \mathscr{I}_i}
	\|\mathcal L h\|_{H^1(  B(x_i, R^*+d))}^2,
	\]
	we have by \eqref{eq 9}
	\[
	\|h\|_{*,\boldsymbol{x}, d} \leq k_0^\frac12\|\mathcal L h\|_{\boldsymbol{x}, d}.
	\]
	Therefore, for $\lambda_i$, $i=1,\dots,k$   such that 
	$\mathcal H\mathcal L h=\mathcal L h -\sum_{i=1}^k\lambda_i\mathcal L \beta_i'(u)$,
	we can obtain 
	\[
	\|h\|_{*,u}\leq  \|h-\sum_{i=1}^k\lambda_i\beta_i'(u)\|_{*,\boldsymbol{x},d}
	\leq  k_0^\frac12  \|\mathcal L h-\sum_{i=1}^k\lambda_i\mathcal L \beta_i'(u)\|_{\boldsymbol{x}, d}
	= k_0^\frac12\|\mathcal H\mathcal L h\|_{\boldsymbol{x}, d},
	\]
	where $\mathcal H=\mathcal H_u$ is the normal projection   from $H_0^1(\A d)$
	to $\set{v\in H_0^1(\A d) | \beta_j'(u)v=0, 1\leq j\leq k}$.
	Let us denote
	\[
	\nabla I^\infty(u) := \mathcal{L} I'(u), \qquad 
	\nabla_{\mathcal{H}} I^\infty(u) := \mathcal{H} \nabla I^\infty(u).
	\]
	Observe that for any \( u \in H_{\boldsymbol{x}, d} \), the support of \( \nabla I^\infty\bigl(\sum_{j=1}^k w_j\bigr) \) is disjoint from \( u_i^\delta \), and hence
	\[
	\nabla_{\mathcal{H}} I^\infty\Bigl(\sum_{j=1}^k w_j\Bigr) = \nabla I^\infty\Bigl(\sum_{j=1}^k w_j\Bigr),
	\quad \text{where } w_j = w(\cdot - x_j).
	\]
	Under this observation, the stability estimate can be reformulated as
	\[
	\Bigl\| u - \sum_{j=1}^k w(\cdot - x_j) \Bigr\|_{\boldsymbol{x}, d}
	\leq C \,
	\Bigl\| \nabla_{\mathcal{H}} I^\infty(u) - \nabla_{\mathcal{H}} I^\infty\Bigl( \sum_{j=1}^k w(\cdot - x_j) \Bigr) \Bigr\|_{\boldsymbol{x}, d}.
	\]
	(ii) 
	If $\sigma(\boldsymbol{x})=0$, then $|x_i-x_j|\geq 2R^*$ for $i\neq j$,
	and the fuction $ \sum w_j$
	is a critical point of the autonomous functional $I^\infty$.
	Regardless of 
	the nonsmoothness,
	by formal calculation  we have 
	\begin{align*}
		\| u - \sum_{j=1}^k w_j \|_{\boldsymbol{x}, d}
		\leq  
		C \| \nabla  I^\infty(u) - \nabla  I^\infty ( \sum_{j=1}^k w_j  )  \|_{\boldsymbol{x}, d} 
		\sim  \|D^2I^\infty(\sum_{j=1}^k w_j)(u-\sum_{j=1}^k w_j)\|_{\boldsymbol{x}, d}.
	\end{align*}
	Consequently, for any $\phi\in \mathcal H H_0^1(\A d)$   we have 
	\[
	\|\phi\|_{\boldsymbol{x}, d}\lesssim  \|D^2I^\infty(\sum w_j)\phi\|_{\boldsymbol{x}, d}.
	\]
	Since $\mathcal H H_0^1(\A d)$ is a subspace of codimension $kN$,
	this inequality expresses the  nondegeneracy of the multi-bump profile  $ \sum w_j$  modulo translations of individual bump $w_j$ —a direct generalization of the nondegeneracy of the single-bump ground state $w$. Thus, the stability result can be viewed as a nonlinear, quantitative extension of this nondegeneracy property to the multi-bump setting.
\end{remark}

Proposition \ref{prop3.1} is a special case of the following result.
\begin{proposition}\label{pro.lim} 
Let  $d_n\in[0, \sigma_0]$, $k_n\in \mathbb N\setminus\{0\}$, and
   $\{u_n\}, \{v_n\}\subset  H_0^1( A(\boldsymbol x_n,d_n))$ be sequences of nonnegative functions that  emerge  around
 $\boldsymbol{x}_n =\{x_j^n\}_{j= 1}^{k_n}  \subset \R^N$   such that  $\sigma(\boldsymbol{x}_n)\to 0$,
 $ \max_{1\leq j\leq k_n}\|u_{n}-  w(\cdot-x_{j}^n)\|_{L^2(B(x_j^n, R^*))}\to 0$,  and 
 $ \|v_{n}-\sum_{j\geq 1} w(\cdot-x_{j}^n)\|_{L^\infty(A(\boldsymbol x_n, 0))}\to 0$.
Suppose that for sequences $\lambda_j^n, \bar\lambda_j^n \in \mathbb{R}^N$, $j=1,\cdots,k_n$ and $h_n, \bar h_n \in H^{-1}(A(\boldsymbol x_n,  d_n))$, the functions $u_n$ and $v_n$ satisfy
\begin{alignat}{2}
  -\Delta u_n - u_n + u_n^{q-1}
  &= \sum_{j=1}^{k_n}\lambda_j^n \cdot (x - x_j^n - \beta_{j}(u_n)) (u_n)_j^\delta + h_n 
  &\quad&\text{in } H^{-1}(A(\boldsymbol x_n,  d_n)), \label{eq55'} \\
  -\Delta v_n - v_n +  v_n^{q-1} 
  &= \sum_{j=1}^{k_n} \bar\lambda_j^n \cdot (x - x_{j}^n - \beta_{j}(v_n)) (v_n)_{j}^\delta + \bar h_n 
  &\quad&\text{in } H^{-1}(A(\boldsymbol x_n,  d_n)), \label{eq56'}
\end{alignat}
where   $\|\bar h_n \|_{*, \boldsymbol x_n,d_n}\to 0$.
Then 
  there exists a constant $C > 0$, independent of $n$, such that
  \[
  \|u_n - v_n\|_{\boldsymbol x_n, d_n} 
  \leq C (\|h_n - \bar h_n\|_{*,\boldsymbol x_n,d_n} +\max_{1\leq j\leq k_n}|\beta_{j}(u_n)-\beta_{j}(v_n)|).
  \]
\end{proposition}
\begin{proof}[Proof of Proposition \ref{pro.lim}]
 Assume by contradiction that
\[
\|h_n - \bar h_n\|_{*,\boldsymbol x_n,d_n}  + \max_{1\leq j\leq k_n} |\beta_j(u_n) - \beta_j(v_n)| = o(1) s_n,
\]
where 
\[
s_n=\|u_n - v_n\|_{\boldsymbol x_n, d_n}> 0.
\]
According to  Lemma \ref{lem2-3} and the assumptions on $\bar h_n$, we have   $\max_{1\leq j\leq k_n}|\bar \lambda_j^n|\to 0$. Therefore,
\begin{equation}\label{lambda}
  \max_{1\leq j\leq k_n}|\lambda_j^n-\bar\lambda_j^n|\leq C\max_{1\leq j\leq k_n}\|u_n-v_n\|_{L^2(B(x_j,\rho))}+o(1) s_n.
\end{equation}
Setting
$z_{n}=\frac{u_n-v_n}{  s_n},$ we have 
$\|z_n\|_{\boldsymbol x_n, d_n}=1$.
By
$|(u_n)_j^{\delta}-(v_n)_j^{\delta}|\leq  |u_n-v_n|$, for each $j$, we have
\begin{equation}\label{22}
\max_{1\leq j \leq k_n}\int \frac{|(u_n)_j^{\delta}-(v_n)_j^{\delta}|}{s_n}   |z_n| \leq \max_{1\leq j \leq k_n}\int_{B(x_j^n,\rho)}  z_n^2 \leq 1.
\end{equation}
Let $i_n \in \{1, \cdots, k_n\}$ be arbitrarily given.
For notational simplicity and
up to a translation, we   denote $i=i_n$. 
Denote
$x_{u_n}^j= x_j^n+\beta_j(u_n)$, $x_{v_n}^j= x_j^n+\beta_j(v_n)$ for $j\in \mathscr I_i$.
By \eqref{eq55'} and \eqref{eq56'}, in $A(\boldsymbol x_n, d_n)\cap B(x_i^n,2R_0)$, we have
\begin{equation}\label{eq57}
\begin{aligned}
-\Delta (u_n-v_n)- (u_n-v_n)&+u_n^{q-1}-v_n^{q-1}=h_n-\bar h_n+\sum_{j\in \mathscr I_i}(\lambda_j^n-\bar\lambda_j^n)\cdot(x-x_{u_n}^j)(u_n)_j^\delta\\
& +\sum_{j\in \mathscr I_i}\bar\lambda_j^n\cdot(x_{v_n}^j-x_{u_n}^j)(v_n)_j^\delta
+ \sum_{j\in \mathscr I_i}\bar\lambda_j^n\cdot(x-x_{u_n}^j)((u_n)_j^\delta-(v_n)_j^\delta) .
\end{aligned}
\end{equation}
Dividing both sides of \eqref{eq57} by $ s_n$, 
 we deduce
\begin{equation}\label{eq58}
\begin{aligned}
-\Delta  z_n-  z_n+  b_n  z_n
=& \sum_{j\in \mathscr I_i}\frac{\lambda_j^n- \bar\lambda_j^n}{s_n}\cdot(x -x_{u_n}^j)(u_n)_j^\delta +\sum_{j\in \mathscr I_i}\bar\lambda_n\cdot\frac{x_{v_n}^j-x_{u_n}^j}{s_n}(v_n)_j^\delta 
\\
&+ \sum_{j\in \mathscr I_i}\bar\lambda_n\cdot(x-x_{u_n}^j)\frac{(u_n)_i^{\delta}-(v_n)_i^{\delta}}{ s_n} +\frac{h_n-\bar h_n}{s_n}\ \ \ \text{in }A(\boldsymbol x_n, d_n)\cap B(x_i^n,2R_0),
\end{aligned}
\end{equation}
where
$$ b_n=\frac{ u_n^{q-1}- v_n^{q-1}}{u_n -v_n}.$$

\noindent
 {\bf Step 1.}  
We show that 
\[ 
\liminf_{n\to+\infty}\max_{1\leq j\leq k_n} \|z_n\|_{L^2(B(x_j^n, \rho))} >0.
\]
Assume to the contrary that $\max_{1\leq j\leq k_n} \|z_n\|_{L^2(B(x_j^n, \rho))} \to 0$
along a subsequence.
Let $\psi$ be a smooth cut-off function satisfying $\psi=1$ in $B(x_i^n, R_0) $,
$\psi=0$ outside $ B(x_i^n, 2R_0)$, $|\nabla \psi|\leq 2/R_0$.
Testing $\psi^2 z_n\in H_0^1(A(\boldsymbol x_n, d_n)\cap B(x_i^n, 2R_0))$ against the equation \eqref{eq58}, 
we have by \eqref{lambda} and \eqref{22} that
\begin{equation}\label{bn}
  \begin{aligned}
      \int|\nabla(\psi z_n)|^2 +\psi^2 b_n z_n^2\leq& \int (|\nabla\psi|^2+\psi^2)z_n^2
  +\sum_{j\in \mathscr I_i}\|z_n\|^2_{L^2(B(x_j^n,R^*+d_n))}
  +o(\|  z_n \|_{H^1(B(x_i^n, 2R_0))})\\
  \leq & C_N\max_{1\leq j\leq k_n}\|z_n\|^2_{L^2(B(x_j^n, R^*+d_n))} +o(1).
  \end{aligned}
\end{equation}
where $C_N$ is a constant depending only on $N$. 
Since  $b_n\geq 0$, by the arbitrary choice of $i=i_n$, the Poincar\'e inequality and \eqref{bn}, we have 
\begin{equation}\label{zn}
1=\|z_n\|^2_{\boldsymbol x_n, d_n} \leq  C\max_{1\leq j\leq k_n}\|z_n\|^2_{L^2(B(x_j^n, R^*+d_n))} +o(1).
\end{equation}
Now 
set $i=i_n$ as the index such that 
\[
 \|z_n\|_{L^2(B(x_i^n, R^*+d_n))}=\max_{1\leq j\leq k_n}\|z_n\|_{L^2(B(x_j^n, R^*+d_n))}.
\]  
 Since 
 $|u|, |v|\leq \delta$ in $ B(x_i^n, R^*+d_n)\setminus B(x_i^n, \rho)$. Therefore, we know 
 \[
 \int_{\RN}b_n\psi^2z_n^2  \geq C\delta^{q-2}\|z_n\|_{L^2(B(x_i^n,R^*+d_n)\setminus B(x_i^n,\rho))}^2.
 \]
 Taking this into \eqref{bn}, we get
\[   
 \int|\nabla(\psi z_n)|^2 + (C\delta^{q-2}-C_N)\|z_n\|_{L^2(B(x_i^n, R^*+d_n))}^2
 \leq C\delta^{q-2}  \|z_n\|_{L^2(B(x_i^n, \rho))}^2+o(1).
\]
 Fixing $\delta$ smaller if necessary, such that $C\delta^{q-2}>C_N+1$, we have 
 by the Poincar\'e inequality that
 \[
 \max_{1\leq j\leq k_n}\|z_n\|_{L^2(B(x_j^n, R^*+d_n))}^2=\|  z_n\|_{L^2(B(x_i^n,R^*+d_n))}^2\leq \|\psi z_n\|_{L^2(\R^N)}^2\leq C_N \|z_n\|^2_{L^2(B(x_i^n, \rho))}+o(1)=o(1).
 \]
This contradicts \eqref{zn}.

\noindent
{\bf Step 2.}   Let $i=i_n$ be the index such that 
\[
 \|z_n\|_{L^2(B(x_i^n, R^*+d_n))}=\max_{1\leq j\leq k_n}\|z_n\|_{L^2(B(x_j^n, R^*+d_n))}.
\]  
Up to translations, we assume $x_i^n=0$.
By Step 1, passing to a subsequence, we may assume $z_n\rightharpoonup z$ in $H^1(B(0, R^*))$ for some
$z\in H^1(B(0, R^*)) \setminus\{0\}$.
We show that 
\begin{equation}\label{step2}
z=\nabla w \cdot \tau \ \text{in } B(0, R^*) \text{ for some } \tau\in \R^N\setminus\{0\}.
\end{equation}
Since
$\sigma(\boldsymbol x_n)\to 0$, we have  $u_n\to w$ a.e. in $B(0, R^*)$, and $v_n\to w$ uniformly in $B(0, R^*)$. Hence,
\begin{equation}\label{eq333}
   b_n(x)=\frac{|u_n|^{q-2}u_n-|v_n|^{q-2}v_n}{u_n -v_n }\to 
 (q-1)w^{q-2}  \ \text{ a.e. in }   B(0,R^*).
\end{equation}
  By this,  $ b_n  \geq 0$, the Fatou's Lemma, and the first inequality in \eqref{bn}, we have 
 \begin{equation}\label{fatou}
 \int_{B(0,R^*)}(q-1)w^{q-2} z^2\leq \liminf_{n\to+\infty}\int \psi^2 b_n   z_n^2\leq C.
 \end{equation}
 
Since  $v_n\to w$ uniformly in $B(0, R^*)$,  for any compact set $ \mathcal K \subset B(0,R_*) $, there exists $ \delta_{\mathcal K} > 0 $ such that $ v_n \geq \delta_{\mathcal K} $ on $ \mathcal K $ for all sufficiently large $ n $. Then, as in \eqref{a.e.}, we have the estimate
\[
\big| |u_n|^{q-2}u_n - |v_n|^{q-2}v_n \big| \leq |v_n|^{q-2}v_n \left| \left| \frac{u_n}{v_n} \right|^{q-1} - 1 \right| \leq C |v_n|^{q-2}v_n \left| \frac{u_n}{v_n} - 1 \right| \leq C \delta_{\mathcal K}^{q-2} |u_n - v_n|.
\]
Therefore, testing \eqref{eq58} with $h\in C_0^\infty(B(0,R^*))$, 
 and taking the limit as $n\to\infty$, we obtain from \eqref{eq333} that
 \begin{equation}\label{eq29}
 -\Delta z-z+(q-1)w^{q-2}z=(\lambda'\cdot x) w^\delta\ \ \ 
 \text{in } B(0,R^*),\  \ \text{where }\frac{\lambda_i^n-\bar\lambda_i^n}{ s_n}\to\lambda'.
 \end{equation}
 By the density of $C_0^\infty(B(0,R^*))$ in the space $H_0^1(B(0,R^*)) \cap L^2(B(0,R^*), w^{q-2} \rd x)$, the identity \eqref{eq29} extends to all test functions $h$ in this space, yielding:
\begin{equation}\label{eq 25}
\int_{B(0,R^*)}\nabla z\nabla h-zh+(q-1)zw^{q-2}h
=\int_{B(0,R^*)}(\lambda'\cdot x)w^\delta h,\ \  \forall h\in H_0^1(B(0,R^*)) \cap L^2(B(0,R^*), w^{q-2} \rd x).
\end{equation}
In view of Lemma \ref{lemmaR*} (i) and (iii), 
\[
 \nabla w \cdot \lambda' = 0 \quad \text{and} \quad \nabla(\nabla w \cdot \lambda') = 0 \quad \text{ on } \partial B(0, R^*).
 \]
Therefore,
$| \nabla w \cdot \lambda'|=o((R^*-|x|))$ as $|x|\to R^*$, and
due to the boundary asymptotics of $w$, we have
\[
\nabla w \cdot \lambda' \in H_0^1(B(0, R^*)) \cap L^2(B(0, R^*), w^{q-2} \rd x).
\]
Since $w$ satisfies $-\Delta w - w + w^{q-1} = 0$ in $\R^N$, differentiating this equation along $\lambda'$ yields $-\Delta(\nabla w \cdot \lambda') - (\nabla w \cdot \lambda') + (q-1)w^{q-2}(\nabla w \cdot \lambda') = 0$ in $\R^N$. 
So we have
\begin{equation*}
\int_{B(0,R^*)}\nabla z\nabla (\nabla w\cdot\lambda')-z(\nabla w\cdot\lambda')+(q-1)zw^{q-2}(\nabla w\cdot\lambda')
=0.
\end{equation*}
Then,
substituting $h = \nabla w \cdot \lambda'$ into \eqref{eq 25}, causes the left-hand side to vanish, which forces
\[
\int_{B(0,R^*)} (\lambda' \cdot x)\, w^\delta (\nabla w \cdot \lambda')  = 0.
\]
According to this and the divergence theorem, we have
 $$0 =\frac12\int _{B(0,R^*)}\text{div}( (w^\delta)^2\lambda')(\lambda'\cdot x)
  =-\frac12\int\nabla (\lambda'\cdot x)\cdot(\lambda'(w^\delta)^2)
  =-\frac12|\lambda'|^2\int_{B(0,R^*)}(w^\delta)^2,$$
 which implies that
 $ \lambda'=0.$ Then by \eqref{fatou} and Lemma \ref{lemmaR*} (iv), we have proved \eqref{step2}.

 \noindent
{\bf Contradiction.} By Lemma \ref{lem2-4}, $|\beta_i(v_n)|\leq \|v_n-w\|_{L^2(B(0,\rho))}=o(1)$, and $\|u_n-v_n\|_{L^2(B(0,\rho))}=o(1)$, we have
 \[ 
\begin{aligned}
\int_{B(0,\rho)}
2xv_n^\delta(u_n- v_n)  
=&\beta_i(u_n)
\|(u_n)_i^\delta\|_{L^2}^2  -\beta_i(v_n) \|(v_n)_i^\delta\|_{L^2}^2  +
O(\|u_n-v_n\|^2_{L^2(B(0,\rho))})\\
=&(\beta_i(u_n)-\beta_i(v_n))\|(u_n)_i^\delta\|_{L^2}^2 + O(|\beta_i(v_n)| \|u_n-v_n\|_{L^2(B(0,\rho))}) +o(1)s_n=o(1)s_n.
\end{aligned} \]
Dividing both sides of this equation by $s_n$ and taking the limit as $n\to\infty$, we obtain
\begin{equation*}
0= \int_{B(0,\rho)}2 xw^\delta  z  
= \int_{B(0,\rho)}2xw^\delta (\nabla w\cdot \tau) 
= \int_{B(0,\rho)}\Big(\nabla\big((w^\delta)^2\big)\cdot\tau\Big)x
=- \tau\int_{B(0,\rho)}(w^\delta)^2,
\end{equation*} 
which contradicts \eqref{step2}.
\end{proof}

\begin{proof}[Proof of Proposition \ref{prop3.1}]
  Assume to the contrary that there exist sequences 
    $d_n\in[0,\sigma_0]$, $\boldsymbol{x}_n=\{x_j^n\}_{j=1}^{k_n}$ 
    with $\sigma(\boldsymbol x_n)\to 0$, and functions $u_n \neq v_n :=\sum_{j=1}^k w(\cdot - x_j^n)$
 emerging around $\boldsymbol{x}_n$
 such that $\beta_j(u_n)=0$, $j=1,\cdots,k_n$ and 
 \(
 \| u_n - v_n  \|_{\boldsymbol{x}_n, d_n} 
\to 0, 
\)
  but 
\[  \|(I^\infty)'(u_n)-(I^\infty)'(v_n)\|_{*,u_n}=o( \|u_n-v_n\|_{\boldsymbol x_n, d_n}).\]
By definition of $\|\cdot\|_{*,u_n}$, 
there exist $\lambda_j^n \in \mathbb{R}^N$, $j=1,\cdots,k_n$, 
such that
\begin{equation}\label{eq3.17}
\|(I^\infty)'(u_n)-(I^\infty)'(v_n)\|_{*,u_n} =\|(I^\infty)'(u_n)-(I^\infty)'(v_n)-\sum_{j=1}^{k_n}\lambda_j^n(x-x_j^n) (u_n)_j^\delta\|_{*,\boldsymbol x_n,d_n}.
\end{equation}
Denote  $h_n=(I^\infty)'(u_n) -\sum_{j=1}^{k_n}\lambda_j^n(x-x_j^n) (u_n)_j^\delta$ and $\bar h_n=(I^\infty)'(v_n)$.
Then we have the following equations in $ H^{-1}(A(\boldsymbol x_n,  d_n))$:
\begin{gather*}
  -\Delta u_n - u_n + |u_n|^{q-2}u_n 
  = \sum_{j=1}^{k_n}\lambda_j^n \cdot (x - x_j^n  ) (u_n)_j^\delta + h_n,
  \\
  -\Delta v_n - v_n +  |v_n|^{q-2}v_n
  = \bar h_n.
\end{gather*}
  By \eqref{Iinfw}, we know 
\[
\|\bar h_n\|_{*,\boldsymbol x_n,d_n}=\|(I^\infty)'(v_n)\|_{*,\boldsymbol x_n,d_n}
\to0.
\]
Combining these with $\beta_j(u_n)=\beta_j(v_n)=0$, $j=1,\cdots,k_n$,
 we get from Proposition \ref{pro.lim} and \eqref{eq3.17} that
 \[
\|u_n - v_n\|_{\boldsymbol x_n, d_n} 
\leq C \|h_n - \bar h_n\|_{*,\boldsymbol x_n,d_n}=C\|(I^\infty)'(u_n)-(I^\infty)'(v_n)\|_{*,u_n}=o(\|u_n - v_n\|_{\boldsymbol x_n, d_n} ),
\]
which is impossible since $u_n\neq v_n$.
 \end{proof}

\section{Solutions with any finite number of bumps}
In Section 4, we construct candidate functions with a finite number of bumps via a constrained minimization procedure. We define an admissible set \(S_{\boldsymbol{x},d}\) of nonnegative functions in \(H^1_0(A(\boldsymbol{x},d))\) that emerge around a given finite point set \(\boldsymbol{x}\in\mathcal K_k\) and satisfy local barycenter constraints. By minimizing the energy \(I(u)\) on this set, we obtain minimizers in Proposition \ref{pro2.5}. This localized construction sets the stage for the final step in Section 5, where we will prove that for \(K\) close to $1$, these minimizers indeed become solutions of \eqref{1.1}.
\subsection{Preliminary}

For $u\in X_q$, we define 
 the energy functional 
\begin{equation*}
I (u)
=\frac12\int_{\RN}|\nabla u|^2-K(x) u^2 dx
+\int_{\RN}  \frac1q|u|^q.
\end{equation*}
By direct computation,
for any $u\in X_q$ that  is emerging around $\boldsymbol x=\{x_i\}_{i=1}^k$, there holds 
\[I(u)=I(u_\delta)+J_\delta(u^\delta),\] where
\begin{equation}\label{9}
 J_\delta(u)= \frac12\int_{\R^N}|\nabla u|^2-K(x)u^2dx
 -\int_{\R^N}K(x)\delta udx+\frac1q\int_{\text{supp }u}(\delta+|u|)^qdx-\frac1q|\text{supp\ }u|\delta^q.
\end{equation}
According to the choice of $\delta$, we know that $ I (u_\delta)>0$ for each $u\geq (\not\equiv) 0$.
Moreover, by the Gagliardo--Nirenberg inequality
\begin{equation}\label{GN}\|u\|_2\leq C_{q,N}\|\nabla u\|_2^\theta\|u\|_q^{1-\theta},\quad u\in X_q, \quad \text{where\ } \frac12=\theta\cdot(\frac12-\frac1N)+(1-\theta)\cdot\frac1q, \end{equation}
we have the following lemma.
\begin{lemma}\label{belowbound}
 For a suitable positive constant $r_1$, it holds that
\begin{align*}
I (u) \geq  \|  u\|_2^2, \mbox{ if }\ u\in X_q\ \text{satisfies } \|u\|_2  \leq r_1.
\end{align*}
\end{lemma}
\begin{proof}
By \eqref{GN} and the Young's inequality, we know that for a constant $C>0$,
\[
 C \|u\|_2^\frac{2N}{N+2\theta} \leq \frac{1}{2}\|\nabla u\|_2^2 +\frac1q\|u\|_q^q.
\]
Therefore,
\[
  I (u)\geq  C\|u\|_2^\frac{2N}{N+2\theta}-a_1\|u\|_2^2=\|u\|_2^2(C\|u\|_2^{-\frac{4\theta}{N+2\theta}}-a_1).
\]
Then the conclusion holds if we fix $r_1>0$ sufficiently small.
\end{proof}
\begin{lemma}\label{lem2.1}
For any $u\in X_q$ emerging around $x\in\R^N$ such that 
$\int_{\R^N}|\nabla u^\delta|^2-K(x)(u^\delta)^2dx<0$, 
the function 
$$\mathcal I_u(t):=I(u_\delta+tu^\delta)$$ has a unique maximum point 
$t_u\in (0,+\infty)$.
\end{lemma}
\begin{proof}
Since
$$
\mathcal I_u(t)=I(u_\delta)+J_\delta(tu^\delta),
$$
we deduce from \eqref{9} that
\begin{equation}\label{eq13}
\mathcal I_u'(t)=\frac{d}{dt}J_\delta(tu^\delta)
=t\int_{\R^N}|\nabla u^\delta|^2-K(x)(u^\delta)^2dx
+\int_{\R^N}(\delta+tu^\delta)^{q-1}u^\delta dx
-\delta\int_{\R^N}K(x) u^\delta dx.
\end{equation}
Therefore, $\mathcal I_u'(0)>0$ and $\mathcal I_u'(t)\to-\infty$
as $t\to+\infty$. Moreover, 
$$\mathcal I_u'''(t)=(q-1)(q-2)\int_{\R^N}(\delta+tu^\delta)^{q-3}(u^\delta)^3<0\ \ \text{for\ all\ }t\in(0,+\infty).$$
So $\mathcal I_u'(t)$ is concave in $(0,+\infty)$
and the conclusion follows.
\end{proof}
\begin{corollary}\label{cor2.2}
 Assume that $u\in X_q$ emerges around the point  $\{x_j\}_{j=1}^k$. Then for any $j\in\{1,\cdots,k\}$, the function $I(u_\delta+\sum_{i\neq j} u_i^\delta+tu_j^\delta)$ has a unique maximum point 
 $t_u^j\in(0,+\infty)$ if $\int_{\R^N}|\nabla u_j^\delta|^2-K(x)(u_j^\delta)^2dx<0$. 
\end{corollary}

For all $\boldsymbol x=\{x_i\}_{i=1}^k\in\mathcal K_k$, we define
$$
S_{\boldsymbol x}=S_{x_1,\dots,x_k}=\Set{u\in X_q\ |\ u\geq 0\ \text{emerging\ around\ }\boldsymbol x,\ I'(u)u_i^\delta=0,\ \beta_i(u)=0\ \text{for each }  i=1,\cdots,k}.$$
According to Lemma \ref{lem2.1}, for 
$u\in S_{x_1,\dots,x_k}$, we have
\begin{gather}
I (u)   = \max_{t_i\geq 0, i=1,\dots,k} I(u_\delta+\sum_{i=1}^kt_iu_i^\delta)=\max_{t\geq 0} I(u_\delta+tu^\delta) > I(u_\delta)>0, \notag\\
J_\delta(u^\delta)= \max_{t_i\geq 0, i=1,\dots,k} J(\sum_{i=1}^kt_iu_i^\delta) = \max_{t\geq 0} J_\delta(tu^\delta) > J_\delta(0)=0.\label{eq 33}
\end{gather}
The following lemma is an immediate consequence of a direct calculation.
\begin{lemma}\label{sss}
There exists a positive constant $\alpha_1$ such that
\[
\sup_{y \in \mathbb{R}^N} \int_{\mathbb{R}^N} |\nabla w_y^\delta|^2 - K(x)(w_y^\delta)^2 \, dx < 0,
\]
provided that $\|K - 1\|_{p,\mathrm{loc}} \leq \alpha_1$, where $w_y = w(\cdot - y)$ for $y \in \mathbb{R}^N$.  
Furthermore, the set $S_{\boldsymbol x}\cap H_0^1(\cup_{j=1}^k B(x_j, R^*))$ is nonempty  for any $k \in \mathbb{N}$ and $\boldsymbol{x} =\{x_i\}_{i=1}^k \in \mathcal{K}_k$.  
\end{lemma}
\begin{proof}
Testing $-\Delta w_y-w_y+w_y^{q-1}=0$ with $w_y^\delta$, we deduce that
  \begin{equation}\label{wy} \int_{\R^N} |\nabla w_y^\delta|^2 - K(x)(w_y^\delta)^2 =\int_{\R^N} (1-K(x)) (w_y^\delta)^2 +\delta w_y^\delta -w_y^{q-1} w_y^\delta
  =\int_{\R^N} (\delta- w^{q-1}) w^\delta +O(\alpha) <0,
    \end{equation}
   where $\|K - 1\|_{p,\mathrm{loc}}=\alpha$ is small.

For  $\boldsymbol{x} =\{x_i\}_{i=1}^k \in \mathcal{K}_k$, the function $u = \max\{w_{x_i} \ |\ i = 1, \ldots, k\}$ emerges around $\boldsymbol{x}$ and satisfies $\int_{\mathbb{R}^N}|\nabla u_j^\delta|^2 - K(x)(u_j^\delta)^2 \, dx < 0$. By Corollary \ref{cor2.2}, we get the conclusion.
\end{proof}
 
Throughout the following proof, we always assume that $\|K - 1\|_{p,loc}\leq \alpha_1$.
\begin{lemma}\label{lem2.3} There is a constant $c>0$ independent of $K\leq a_1$,
  $k\in\N\setminus\{0\}$ and   $\{x_i\}_{i=1}^k\in\mathcal K_k$,  such that, for all $u\in S_{x_1,\dots,x_k}$, it holds
$$\|u_i^\delta\|_{q}\geq c,\quad \int_{\R^N}|\nabla u_i^\delta|^2-K(x)(u_i^\delta)^2dx<-c,\quad  i=1,\cdots,k.$$
\end{lemma}
\begin{proof}
By $I '(u)u_i^\delta=J_\delta'(u_i^\delta)u_i^\delta=0$, it holds
\begin{equation}\label{eq10}
\begin{aligned}
\int_{\R^N}|\nabla u_i^\delta|^2-K(x)(u_i^\delta)^2dx
&=-\int_{\R^N}(\delta+u_i^\delta)^{q-1}u_i^\delta dx
+\delta\int_{\R^N}K(x) u_i^\delta dx\\
&\leq -\int_{\R^N}( t_* \delta+(1-t_*) u_i^\delta)^{q-1}u_i^\delta dx
+\delta\int_{\R^N}a_1 u_i^\delta dx\\
&\leq \int_{\R^N}(\delta a_1-t_*\delta^{q-1}) u_i^\delta dx-(1-t_*)\int_{\R^N}(u_i^\delta)^{q}dx\\
&\leq -(1-t_*)\int_{\R^N}(u_i^\delta)^{q}dx,
\end{aligned}
\end{equation}
where we used the fact $t_*\in(0,2-q)$, 
the concavity  of the function $s\mapsto s^{q-1}$, $s\geq 0$, and $\delta<(t_*/a_1)^\frac{1}{2-q}$ (see \eqref{eq9}).
By this and \eqref{GN}, we get
$$
\begin{aligned}
\|u_i^\delta\|_{q}^q\leq C_{q,N}\|u_i^\delta\|_{q}^2,
\end{aligned}
$$
which implies that $\|u_i^\delta\|_{q}\geq c>0$. Taking this into \eqref{eq10}, we get the conclusion.
\end{proof}

\begin{lemma}\label{lem2.4}
  There is a constant $c_0>0$,  such that  $I (u_\delta+\sum_{j=1}^k (1+s_j)u_j^\delta)
\leq I (u)-c_0\sum_{j=1}^k\min\set{s_j^2,1}$ 
for any $u\in S_{\boldsymbol x}$ with $\boldsymbol x=\{x_i\}_{i=1}^k$. 
\end{lemma}
\begin{proof}
Note that
$$I (u_\delta+\sum_{j=1}^k (1+s_j)u_j^\delta)=I(u_\delta)+\sum_{j=1}^k J_\delta((1+s_j)u_j^\delta).$$
  For the function
$
\mathcal I_{u,j}(t):=J_\delta(tu_j^\delta)
$, we have \eqref{eq13} and
$$
\begin{aligned}
&\mathcal I_{u,j}''(t)=\int_{\R^N}|\nabla u_j^\delta|^2-K(x)(u_j^\delta)^2dx
+\int_{\R^N}(q-1)(\delta+tu_j^\delta)^{q-2}(u_j^\delta )^2dx.
\end{aligned}
$$
Since $\mathcal I_{u,j}'(1)=0$, $t_*\in(0,2-q)$, 
 $s\mapsto s^{q-1}$ is concave for $s\geq 0$, and $\delta<(t_*/a_1)^\frac{1}{2-q}$, we obtain
$$
\begin{aligned}
\mathcal I_{u,j}''(t)=&\int_{\R^N}|\nabla u_j^\delta|^2-K(x)(u_j^\delta)^2dx
+\int_{\R^N}(q-1)(\delta+tu_j^\delta)^{q-2}(u_j^\delta )^2dx\\
=&\int_{\R^N}(q-1)(\delta+tu_j^\delta)^{q-2}(u_j^\delta )^2 
-(\delta+u_j^\delta)^{q-1}u_j^\delta 
+\delta K(x) u_j^\delta dx\\
\leq&\int_{\R^N}(q-1)t^{q-2}(u_j^\delta)^{q} 
-(t_*\delta+(1-t_*) u_j^\delta)^{q-1}u_j^\delta 
+\delta a_1 u_j^\delta dx\\
\leq&\int_{\R^N}(q-1)t^{q-2}(u_j^\delta)^{q} 
-(1-t_*)(u_j^\delta)^{q}-t_*\delta^{q-1}u_j^\delta
+\delta a_1 u_j^\delta dx\\
\leq &\int_{\R^N}((q-1)t^{q-2}-(1-t_*))(u_j^\delta)^{q}dx.
\end{aligned}
$$
By Lemma \ref{lem2.3}, there is $c>0$ such that 
  $\mathcal I_{u,j}''(t)\leq -c$ for each $t$ in a small  interval that is independent of $u$ and contains $1$. Recalling that, for $t\in(0,+\infty)$,
$\mathcal I_{u,j}'(t)$ has a unique zero  at $t=1$ and $\mathcal I_{u,j}'''(t)<0$, we have
$J_\delta((1+s_j)u_j^\delta)\leq J_\delta(u_j^\delta)-c_0\min\set{s_j^2,1}$. Combining this with the fact
$I(v)=I(v_\delta)+\sum_{j=1}^k J_\delta(v_j^\delta)$,
we complete the proof.
\end{proof}
\begin{remark} 
According to the proof of the above lemma, we know that $\mathcal I_{u,j}''(t)<0$ if $\mathcal I_{u,j}'(t)=0$ with $t>0$. By the implicit function theorem, $t_u^j$ given in Corollary \ref{cor2.2} is continuous in $u$.
\end{remark}
\begin{lemma}\label{lem 4.8}
Let $t_u$ be defined in Lemma \ref{lem2.1}.
Then $\{t_{w_y}\ |\ y\in\R^N\}$ is bounded, where $w_y= w(\cdot -y)$.
\end{lemma}
\begin{proof}
  By \eqref{eq13}, 
  \[
    t_{w_y}\int_{\R^N}|\nabla w_y^\delta|^2-K(x)(w_y^\delta)^2dx
    +\int_{\R^N}(\delta+t_{w_y}w_y^\delta)^{q-1}w_y^\delta dx
    -\delta\int_{\R^N}K(x) w_y^\delta dx=0  
  \]
  Since 
  $(\delta+t_{w_y}w_y^\delta)^{q-1}\leq  \delta^{q-1}+t_{w_y}^{q-1}(w_y^\delta)^{q-1}$, we have by \eqref{wy}
  \[
    t_{w_y}^{q-1}\|w^\delta\|_q^q \geq (\delta- \delta^{q-1}+O(\alpha))\|w^\delta\|_{1}  +t_{w_y}(\int_{\R^N} (  w^{q-1}-\delta) w^\delta +O(\alpha)).
  \]
  The proof is complete for $\alpha=\|K(x)-1\|_{p,loc}\leq \alpha_1$ since $q-1<1$ and $\int_{\R^N} (  w^{q-1}-\delta) w^\delta+O(\alpha)>0$.
\end{proof}
 \begin{lemma}\label{lem2.5}
 For all $k\in\N\setminus\{0\}$ and all $\boldsymbol x=\{x_i\}_{i=1}^k\in\mathcal K_k$, if $ u\in S_{\boldsymbol x}$ such that $I (u)\leq Ck,$ then $\|u_\delta\|+\sum_{i=1}^k\|u_i^\delta\|\leq C'k^{\theta'}$, where $C',\theta'$ are independent of $u$, $k$, $\boldsymbol x$, and $K$.
 \end{lemma}
 \begin{proof}
 By \eqref{eq 33} and \eqref{eq9}: 
 $\delta<\left(\frac{t_*}{a_1}\right)^\frac{1}{2-q}<\left(\frac{1}{qa_1}\right)^\frac{1}{2-q}$,
 we have
 \begin{equation*} 
 Ck\geq I(u)=I(u_\delta)+J_\delta(u^\delta)
 >I (u_\delta)+J_\delta(0)\geq \frac12\|\nabla u_\delta\|_2^2+\frac{1}{2q} \|u_\delta\|_q^q.
 \end{equation*}
 So $\|u_\delta\|\leq Ck$. 
 On the other hand, since $I(u_\delta)>0$, for $\mathcal S_{i}=\text{supp }u_i^\delta$, we have
 \begin{equation*} 
 \begin{aligned}
 Ck\geq J_\delta(u_i^\delta)-\frac12J_\delta'(u_i^\delta)u_i^\delta
 =&-\frac12\delta\int_{\R^N}K(x) u_i^\delta dx+\frac1q\int_{\mathcal S_{i}}(\delta+u_i^\delta)^qdx-\frac1q|\mathcal S_{i}|\delta^q-\frac12\int_{\R^N}(\delta+u_i^\delta)^{q-1}u_i^\delta dx\\
 \geq&(\frac1q-\frac12)\int_{\R^N}(u_i^\delta)^qdx
 -\frac12\int_{\R^N}\delta^{q-1}u_i^\delta +\delta K(x) u_i^\delta dx\\
 \geq& C_q\|u_i^\delta\|_q^q-C_{\rho}\|u_i^\delta\|_q.
 \end{aligned}
 \end{equation*}
 Thus $\|u_i^\delta\|_q\leq Ck^{\theta''}$ for some $\theta''\geq 1$. By the  Gagliardo--Nirenberg inequality
 $$\|u_i^\delta\|_2\leq C_{q,N}\|\nabla u_i^\delta\|_2^\theta\|u_i^\delta\|_q^{1-\theta},\ \ \ \text{where\ } \frac12=\theta\cdot(\frac12-\frac1N)+(1-\theta)\cdot\frac1q,$$
and the Young inequality, we deduce
 \begin{equation*} 
 \begin{aligned}
 Ck\geq J_\delta(u_i^\delta)\geq \frac12\|\nabla u_i^\delta\|_2^2-\frac{a_1}{2}\|u_i^\delta\|_2^2-C_\rho\geq
 \frac14\|\nabla u_i^\delta\|_2^2-C\|u_i^\delta\|_q^2-C_\rho.
 \end{aligned}
 \end{equation*}
 Thus $\|\nabla u_i^\delta\|_2\leq C'k^{\theta'}$, and the proof is completed.
 \end{proof}

\subsection{Local minimizer}
For $k\in\N\setminus\{0\}$, $\boldsymbol x=\{x_j\}_{j=1}^k\in\mathcal K_k$ and
$d\geq 0$,
recall that
\begin{align*} A(\boldsymbol x,d):&=  \bigcup_{j=1}^kB(x_j,R^*+d),\end{align*}
and define
\begin{align*}
S_{\boldsymbol{x}, d}:& =S_{\boldsymbol{x}} \cap H_0^1(\A d).
\end{align*}
We consistently regard  $H_0^1(\A d)$ as a subspace of $X_q$ by extending functions as zero outside
$\A d$.  On this subspace, the norm $\|\cdot\|_{H^1(\A d)}$ is equivalent to $\|\cdot\|$ due to the Sobolev embedding theorem.
By Lemma \ref{sss}, we know $S_{\boldsymbol{x}, d}\neq \emptyset$.
We also define
 \begin{equation}\label{mux}
  \mu_d(\boldsymbol x)=\mu_d(x_1,\dots,x_k)=\inf_{u\in S_{\boldsymbol x, d}} I (u), \quad \mu (\boldsymbol x)=\mu (x_1,\dots,x_k)=\inf_{u\in S_{\boldsymbol x }} I (u). 
 \end{equation}
 We will demonstrate that this minimization problem admits a minimizer, and denote the set of minimizers of \eqref{mux} by
$$M_{\boldsymbol x,d}=M_{x_1,\dots,x_k,d}=\{u\in S_{\boldsymbol x, d}\ |\ I (u)=\mu_d(\boldsymbol x)\},$$
$$M_{\boldsymbol x }=M_{x_1,\dots,x_k }=\{u\in S_{\boldsymbol x }\ |\ I (u)=\mu (\boldsymbol x)\}.$$
Then we define 
\[
  \mu_{k,d}=\sup_{\boldsymbol x\in\mathcal K_k}\mu_d(\boldsymbol x),\quad 
  \mu_{k }=\sup_{\boldsymbol x\in\mathcal K_k}\mu (\boldsymbol x).
\]
Analogously, 
we define 
\[
  \mu^\infty(\boldsymbol x)=\inf_{u\in S_{\boldsymbol x}^\infty} I^\infty(u), \quad  \mu_d^\infty(\boldsymbol x)=\inf_{u\in S_{\boldsymbol x, d}^\infty} I^\infty(u),
\]
$$M_{\boldsymbol x}^\infty =\{u\in S_{\boldsymbol x}^\infty\ |\ I^\infty(u)=\mu^\infty(\boldsymbol x)\},
\quad
M_{\boldsymbol x, d}^\infty =\{u\in S_{\boldsymbol x, d}^\infty\ |\ I^\infty(u)=\mu^\infty(\boldsymbol x, d)\},
$$
and 
\[
  \mu_k^\infty=\sup_{\boldsymbol x\in\mathcal K_k}\mu^\infty(\boldsymbol x),  \quad \mu_{k,d}^\infty=\sup_{\boldsymbol x\in\mathcal K_k}\mu_d^\infty(\boldsymbol x),
\]
where $S_{\boldsymbol x}^\infty$ is defined as the set $S_{\boldsymbol x}$ for $K(x)\equiv1$,
and $S_{\boldsymbol x, d}^\infty= S_{\boldsymbol x}^\infty \cap H_0^1(\A d)$.
It is clear that $\mu_d(\boldsymbol{x}), \mu_d^\infty(\boldsymbol{x}), \mu_{k,d}, \mu_{k,d}^\infty$  are
decreasing with respect to $d\geq 0$ and 
$\mu_d(\boldsymbol{x})\geq \mu(\boldsymbol{x})$, $\mu_d^\infty(\boldsymbol{x})\geq \mu^\infty(\boldsymbol{x})$,
 $\mu_{k,d}\geq \mu_{k}$, and $\mu_{k,d}^\infty\geq \mu_{k}^\infty$.

\begin{proposition}\label{pro2.5}
For $k \in \mathbb{N}\setminus\{0\}$, $\boldsymbol{x} = \{x_i\}_{i=1}^k \in \mathcal{K}_k$, and $d \geq 0$, if a sequence $\{u_n\} \subset S_{\boldsymbol{x}, d}$ satisfies $I(u_n) \to \mu_d(\boldsymbol{x})$, then it contains a convergent subsequence. Moreover, $\mu_d(\boldsymbol{x})$ is attained by some $u \in S_{\boldsymbol{x}, d}$.
\end{proposition}
\begin{proof}
Set $v=\max_{i\in\{1,\dots,k\}}\{w_{x_i}\}$. Then $v_i^\delta=(w_{x_i})^\delta$. By Corollary \ref{cor2.2}, Lemma \ref{sss} and Lemma \ref{lem 4.8}, we have 
  \[ 
    \mu_d(\boldsymbol x) \le I(v_\delta+ \sum_{i=1}^kt_{w_{x_i}}(w_{x_i})^\delta)\leq Ck.
  \]
Let   $\{u_n\}\subset S_{\boldsymbol x,d}$ be a minimizing sequence. Then by Lemma \ref{lem2.5}, 
$\|u_n\|\leq Ck^{\theta'}$.
Up to a subsequence, we assume $u_n\rightharpoonup u$ in $X_q\cap H_0^1(\A d)$,
$u_n\to u$ in $L^r_{loc}(\R^N)$ with $r\in[1,2^*)$ and $u_n(x)\to  u(x)$ a.e. in $\R^N$. Thus we have $ u(x)\leq \delta$ in
$\R^N\setminus B(x_i,\rho)$, $\beta_i(u)=0$, and by Lemma \ref{lem2.3},
\begin{equation}\label{eq14}
\int_{\R^N}|\nabla u_i^\delta|^2-K(x)(u_i^\delta)^2dx
\leq \liminf_{n\to\infty}
\int_{\R^N}|\nabla (u_n)_i^\delta|^2-K(x)((u_n)_i^\delta)^2dx\leq -c<0.
\end{equation}
By Corollary \ref{cor2.2} and \eqref{eq14}, there exist
$t_i>0$, such that $\overline u:=u_\delta+\sum_{i=1}^k t_i u_i^\delta\in S_{\boldsymbol x,d}$. Then $\overline u$ is the minimizer, since
$$I(\overline u)
= I(u_\delta+\sum_{i=1}^k t_i u_i^\delta)
\leq \liminf_{n\to\infty}
I((u_n)_\delta+\sum_{i=1}^k t_i (u_n)_i^\delta)
\leq \liminf_{n\to\infty}I(u_n)-c_0\sum_{i=1}^k\min\set{(t_i-1)^2,1}
\leq\mu_d(\boldsymbol x),$$
where we have used  Lemma \ref{lem2.4}. 
Therefore, $I(\overline u)=\mu_d(\boldsymbol x)$, and thus the above inequalities hold with equality and $t_i = 1$ for $i=1,\cdots,k$. In particular, $I(u)=\liminf_{n\to \infty} I(u_n)$.
Recalling that
$u_n\to u$ in $L^r_{loc}(\R^N)$ with $r\in[1,2^*)$ and that $\{u_n\}\subset  H_0^1(\A d)$, we get the conclusion that
$\{u_n\}$ contains a convergent subsequence.
\end{proof}

\begin{lemma}\label{lem2.10}
  Let $k \in \mathbb{N}\setminus\{0\}$, $ \boldsymbol x =\{x_i\}_{i=1}^k \in\mathcal K_k$  and $d\geq 0$.
If  $v\in M_{\boldsymbol{x}, d}$   satisfies $\text{supp}(v^\delta)\subset \bigcup_{i=1}^k\overline {B(x_i,\rho_0)}$ for some $\rho_0\in(0,\rho]$,
then $v$ satisfies
\begin{equation} \label{26}
        -\Delta v-K(x)v+ v^{q-1} =0, \quad 0\leq v<\delta,  \quad \text{in\ }\A d\setminus \bigcup_{i=1}^k\overline {B(x_i,\rho_0)}.
\end{equation}
\end{lemma}
\begin{proof}
Set
$\mathcal X:=\{u\in H_0^1(\A d)\ |\   |u|\leq\delta\ \text{in }\A d\setminus \bigcup_{i=1}^k\overline {B(x_i,\rho_0)},\ u=v\ \text{in\ } \bigcup_{i=1}^k\overline {B(x_i,\rho_0)}\}$. Then 
$$I(u)=I(u_\delta)+J_\delta(u^\delta)
 =I (u_\delta)+J_\delta(v^\delta), \ \ u\in\mathcal X.$$
Easily to check, the functional $I $ is bounded from below, continuous, strictly convex and satisfies
$I(u)\geq \frac12\|\nabla u_\delta\|_2^2+\frac{1}{2q}\| u_\delta\|_q^q$ on the closed convex set $\mathcal X$.
Hence $\inf\{I(u)\ |\ u\in\mathcal X\}$ has a unique minimizer $u\in\mathcal X$. By this and $I(|u|)= I(u)$, we know that $u=|u|\in \mathcal X$ and $0\leq u \leq \delta$ in $\A d\setminus \text{supp}(v^\delta)$.  
By Proposition \ref{pro2.5} and $v\in M_{\boldsymbol{x}, d}$, we know that $v$ is a minimizer. According to the uniqueness, $u=v$.

Take any $h\in C_0^\infty(\A d\setminus \bigcup_{i=1}^k\overline {B(x_i,\rho_0)})$ with $h\geq 0$. We know $v-t h \in \mathcal X$ for small $t>0$.
Hence by $I(v-th)\geq I(v)$ for small $t>0$, we have 
$I'(v)h\leq 0$.
It means $v$ is a subsolution of \eqref{26}.
Recalling the choice of $\delta$ in \eqref{eq9}, we have
\[
-\Delta v  \leq 0,\quad 0\leq v\leq \delta\mbox{ in } \A d\setminus \bigcup_{i=1}^k\overline {B(x_i,\rho_0)}.
\]
By this and the strong maximum principle,
we know that $v<\delta$ in $\A d\setminus \bigcup_{i=1}^k\overline {B(x_i,\rho_0)}$.
Therefore, the $\inf\{I (u)\ |\ u\in\mathcal X\}$ is attained at $v$, which is inside $\mathcal{X}$. In other words, for any $h\in C_0^\infty(\A d\setminus \bigcup_{i=1}^k\overline {B(x_i,\rho_0)})$,  $v+t h \in \mathcal X$ for small $t\in\R$, and hence
$I(v+th)\geq I(v)$ for small $t\in\R$, which means that
 $v$ is a solution of \eqref{26}.
\end{proof}
\begin{remark}\label{rek4.2} (i) From the proof of Lemma \ref{lem2.10},
   we know that for each $v$ emerging around $\boldsymbol{x}\in \mathcal K_k$, and each $R\in[\rho, R^*)$,
   the following equation admits a unique solution
  \[  
  \begin{cases}
  -\Delta u-K(x)u+ u^{q-1} =0, \quad   0\leq u<\delta   \ \text{ in  } A(\boldsymbol{x}, d) \setminus \bigcup_{i=1}^k \overline{B(x_i,R)},\\
u \in H_0^1(A(\boldsymbol{x}, d)), \quad u= v \ \text{ in } \bigcup_{i=1}^k\overline{B(x_i,R)}.
\end{cases}
\]
The existence of a solution follows from the attainability of the corresponding minimization problem, and uniqueness is a consequence of the strict convexity of the functional.

(ii)
By \eqref{26},  and $K\in L^p_{loc}(\RN)$ with $p> N/2\geq 1$, we have $v\in W^{2,p}(\A d)\hookrightarrow W^{1,p}(\partial \A d)$. Since $v\geq 0$ inside $\A d$, we can check that $\frac{\partial v}{\partial \nu}\leq 0$ a.e. on $\partial\A d$, where $\nu$ is the outer  normal to $\partial\A d$. Therefore, $v$ weakly satisfies
  \[  
        -\Delta v-K(x)v+ v^{q-1} \leq 0, \quad 0\leq v<\delta,  \quad \text{in\ }\RN \setminus \bigcup_{i=1}^k\overline {B(x_i,\rho )}.
\]
Then, by Lemma \ref{R0},
\begin{equation}\label{eq28}
\supp v \subset \bigcup_{i=1}^k B(x_i, R_0)=\A { \sigma_0^\frac{2-q}{2}}.
\end{equation}
 \end{remark}
  We use notation $a\vee b=\max\{a,b\}$ and $a\wedge b=\min\{a,b\}$.
  Since $a\vee b=a$ if and only if $a\wedge b=b$,
  for $u,v\in X_q$, it holds
  \[
I (u_1\vee u_2)  =I (u_1)+I  (u_2) -I(u_1\wedge u_2).
\]
It is clear that if $u\in S_{\boldsymbol{y}}$, $v\in S_{\boldsymbol{z}}$ with 
$\boldsymbol{y}\in \mathcal K_{k_1}$, 
$\boldsymbol{z}\in \mathcal K_{k_2}$, and
$\dist(\boldsymbol{y},\boldsymbol{z})\geq 2R^*-\sigma_0$,
then $\boldsymbol{x}=\boldsymbol{y}\cup\boldsymbol{z}\in \mathcal K_{k_1+k_2}$, and $u\vee v\in S_{\boldsymbol{x}}$.
The following result holds by \eqref{eq28}:
\begin{corollary}\label{cor413}
Let \( k_1, k_2 \geq 1 \) and \( d \geq 0 \).  
Assume that \( \boldsymbol{x} \in \mathcal{K}_{k_1} \) and \( \boldsymbol{y} \in \mathcal{K}_{k_2} \) satisfy  
\( \boldsymbol{x} \cup \boldsymbol{y} \in \mathcal{K}_{k_1 + k_2} \).  
Then
\[
\mu_d(\boldsymbol{x} \cup \boldsymbol{y}) \leq \mu_d(\boldsymbol{x}) + \mu_d(\boldsymbol{y}).
\]
If, in addition,
\[
\min\bigl\{ |x - y| \;\big|\; x \in \boldsymbol{x},\, y \in \boldsymbol{y} \bigr\} \geq 2R_0,
\]
then equality holds:
\[
\mu_d(\boldsymbol{x} \cup \boldsymbol{y}) = \mu_d(\boldsymbol{x}) + \mu_d(\boldsymbol{y}).
\]
\end{corollary}

\begin{proof}
Let \( v_1 \in M_{\boldsymbol{x}, d} \) and \( v_2 \in M_{\boldsymbol{y}, d} \).  
Then \( v_1 \vee v_2 \in S_{\boldsymbol{x} \cup \boldsymbol{y}, d} \), and hence
\[
\mu_d(\boldsymbol{x} \cup \boldsymbol{y}) 
\leq I(v_1 \vee v_2)
= I(v_1) + I(v_2) - I(v_1 \wedge v_2)
\leq \mu_d(\boldsymbol{x}) + \mu_d(\boldsymbol{y}),
\]
where we used that \( v_1 \wedge v_2 \leq \delta \), which implies \( I(v_1 \wedge v_2) \geq 0 \).

Now assume that  
\( \min\{ |x - y| \mid x \in \boldsymbol{x},\, y \in \boldsymbol{y} \} \geq 2R_0 \).  
Let \( u \in M_{\boldsymbol{x} \cup \boldsymbol{y}, d} \).  
By \eqref{eq28}, we can write \( u = u_1 + u_2 \) with  
\( u_1 \in S_{\boldsymbol{x}, d} \), \( u_2 \in S_{\boldsymbol{y}, d} \), and  
\( \operatorname{supp} u_1 \cap \operatorname{supp} u_2 = \emptyset \).  
Therefore,
\[
\mu_d(\boldsymbol{x} \cup \boldsymbol{y}) = I(u) = I(u_1) + I(u_2) \geq \mu_d(\boldsymbol{x}) + \mu_d(\boldsymbol{y}).
\]
Combining this with the previous inequality,  we arrive at the desired equality.
\end{proof}
\begin{lemma}\label{lem mu}
 $\mu_d(\boldsymbol{x})$ is continuous  with respect to $d$ and $\boldsymbol{x}$.
\end{lemma}
\begin{proof}
Consider sequences $d_n\to d$ and $\boldsymbol{x_n}\to \boldsymbol{x}$ in $\mathcal K_k$, with
$\boldsymbol{x_n}=\{x_i^n\}_{i=1}^k$, $\boldsymbol{x}=\{x_i \}_{i=1}^k$. 
Let $u\in M_{\boldsymbol{x}, d}$. 
Fix  $R\in(\rho, \rho+\sigma_0)$.
By Remark \ref{rek4.2}, for each 
$n$, there exists a unique function 
$u_n$ satisfying
\begin{equation} \label{un}
  \begin{cases}
  -\Delta u_n-K(x)u_n+ u_n^{q-1} =0, \quad   0\leq u_n<\delta \ \text{ in  } A(\boldsymbol{x_n}, d_n) \setminus \bigcup_{i=1}^k\overline {B(x_i,R)},\\
u_n \in H_0^1(A(\boldsymbol{x_n}, d_n)), \quad u_n= u(\cdot-x_i^n+x_i)\  \text{ in } \overline{B(x_i,R)},\quad i=1,\cdots,k.
\end{cases}
\end{equation}
Let $\varphi_i$ be a smooth cut-off function such that $\varphi_i=1$ in $B(x_i,R)$ and
 $\varphi_i=0$ outside $B(x_i,\rho+\sigma_0)$. We define $\tilde u_n=\sum_{i=1}^k\varphi_i u(\cdot-x_i^n+x_i)$. Then, for large $n$, it holds $u_n-\tilde u_n\in H_0^1(A(\boldsymbol{x_n}, d_n) \setminus \bigcup_{i=1}^k\overline {B(x_i,R)})$, and $u_n, \tilde u_n\in[0,\delta)$ in $A(\boldsymbol{x_n}, d_n) \setminus \bigcup_{i=1}^k\overline {B(x_i,R)}$. Therefore,
testing equation \eqref{un} with $u_n-\tilde u_n$, we obtain
\[
\begin{aligned}
\int_{A(\boldsymbol{x_n}, d_n) \setminus \bigcup_{i=1}^k\overline {B(x_i,R)}}
\nabla u_n(\nabla u_n-\nabla \tilde u_n)\leq C.
\end{aligned}
\]
This, together with the Young inequality and the fact $\|\nabla \tilde u_n\|_2\leq C$, implies that 
$$\int_{A(\boldsymbol{x_n}, d_n) \setminus \bigcup_{i=1}^k\overline {B(x_i,R)}}
|\nabla u_n|^2\leq C.$$
By this, $u_n=u(\cdot-x_i^n+x_i)$ in $B(x_i,R)$, $u_n\in[0,\delta)$ in $A(\boldsymbol{x_n}, d_n) \setminus \bigcup_{i=1}^k\overline {B(x_i,R)}$, and $u_n\in H_0^1(A(\boldsymbol{x_n}, d_n))$, we
 get the boundedness of
$\{u_n\}$ in $X_q$. So up to a subsequence, we can assume that
$u_n\rightharpoonup v$ weakly in  $X_q$.
Since $u_n\in H_0^1(A(\boldsymbol{x_n}, d_n))$, we have $v\in H_0^1(\A d)$.
Testing the equation \eqref{un} with any   $h\in C_0^\infty(A(\boldsymbol{x}, d) \setminus \bigcup_{i=1}^k\overline {B(x_i,R)})$ and passing to the limit as $n\to\infty$, we find that $v$ satisfies
\[  
  \begin{cases}
  -\Delta v-K(x)v+ v^{q-1} =0, \quad   0\leq v<\delta   \ \text{ in  } A(\boldsymbol{x}, d) \setminus \bigcup_{i=1}^k\overline {B(x_i,R)},\\
v \in H_0^1(A(\boldsymbol{x}, d)), \quad v= u \ \text{ in } \overline{B(x_i,R)},\quad i=1,\cdots,k.
\end{cases}
\]
Since $u\in M_{\boldsymbol{x}, d}$, it solves this equation. According to Remark \ref{rek4.2}, it is
the unique solution, and hence $v=u$.
Next, we prove the strong convergence $u_n\to u$ in $X_q$.

Set $\delta_0=\frac12(R-\rho)$.
By Lemma \ref{lem2.10} and the elliptic regularity theory, we have
\begin{equation}\label{eq4.14}
u\in W^{2,p}(\A d\setminus \cup_{j=1}^k \overline{B(x_j, R-\delta_0)}).
\end{equation}
For any $\e>0$,  we choose $u_\e\in H_0^1(\Omega_\e)$ with $\Omega_\e$ an open smooth set  satisfying
\[\bigcup_{j=1}^{k} B(x_j, \rho+\sigma_0)\subset\Omega_\e \subset \overline\Omega_\e \subset\A d,\ \
\text{and } \|u_\e-u\|\leq \e.\]
Then for large $n$, we have $\overline\Omega_\e\subset A(\boldsymbol{x}_n, d_n)$.
By this, the smoothness of $\partial B(x_i,R)$ and \eqref{eq4.14}, we can use the $L^p$ estimate in \eqref{un}, and obtain that $\|u_n\|_{W^{2,p}(\Omega_\e \setminus \cup_{j=1}^{k}\overline{ B(x_j, R)})}$ is bounded. Since $p>\frac{N}{2}\geq \frac{2N}{N+2}$, by  the  Rellich theorem,  and the fact $u_n\rightharpoonup v$ weakly in  $X_q$, we have 
$u_n\to u$ in $H^1(\Omega_\e\setminus \cup_{j=1}^{k} \overline{B(x_j, R)} )$. 
Combining this with $u_n=u(\cdot-x_i^n+x_i)\to u$ in $H^1(\cup_{j=1}^k B(x_j,R))$, we deduce that
$u_n\to u$ in $H^1(\Omega_\e)$. 
Since 
$u_n$ and $u$ 
are supported in a compact set indendent of $n$, we have 
$u_n\to u$  strongly in $L^r(\R^N)$, $r\in[1,2^*)$.
Consequently, for large $n$,
\begin{equation}\label{eq4.15}
\begin{aligned}
\|u_n-u_\e\|_{L^r(\R^N)}+\|u_n-u_\e\|_{H^1(\Omega_\e)}= \|u-u_\e\|_{L^r(\R^N)}+\|u-u_\e\|_{H^1(\Omega_\e)}+o_n(1)\leq C_r\e,
\end{aligned}
\end{equation}
where $C_r$ is a constant depending on $r\in[1,2^*)$.

Now let 
$\psi$ be a smooth nonnegative cut-off function such that 
$\psi=0$ in $\cup_{j=1}^{k} B(x_j, R)$ and 
$\psi=1$ in $\R^N\setminus\cup_{j=1}^{k} B(x_j, \rho+\sigma_0)$.
Then testing \eqref{un} with $\psi (u_n-u_\e)\in H_0^1(A(\boldsymbol{x}_n, d_n)\setminus \cup_{j=1}^{k} B(x_j, R))$, we obtain by \eqref{eq4.15}
\begin{align*}
\int_{A(\boldsymbol{x}_n, d_n)} \nabla u_n\nabla(\psi(u_n-u_\e))=&
\int_{A(\boldsymbol{x}_n, d_n)} \psi(Ku_n-u_n^{q-1})(u_n-u_\e)\\
\leq& C\|K\|_{p,loc}\|u_n-u_\e\|_{L^\frac{p}{p-1}}
+C\|u_n-u_\e\|_{L^q}\leq C\e.
\end{align*}
By this and \eqref{eq4.15}, we have that, for large $n$,
\[
\begin{aligned}
\int_{A(\boldsymbol{x}_n, d_n)\setminus\Omega_\e}|\nabla u_n|^2
&=\int_{A(\boldsymbol{x}_n, d_n)\setminus\Omega_\e} \nabla u_n\nabla(\psi(u_n-u_\e))\\
&= \int_{A(\boldsymbol{x}_n, d_n)} \nabla u_n\nabla(\psi(u_n-u_\e))-
 \int_{\Omega_\e} \nabla u_n\nabla(\psi(u_n-u_\e))\\
&\leq C\e,
\end{aligned}
\]
and hence
\[
\|\nabla (u_n-u_\e)\|^2_{L^2(A(\boldsymbol{x}_n, d_n)\setminus \Omega_\e)}
=\|\nabla u_n \|^2_{L^2(A(\boldsymbol{x}_n, d_n)\setminus \Omega_\e)} \leq C\e.
\]
According to this and \eqref{eq4.15}, we obtain
\[
\limsup_{n\to+\infty}\|u_n-u\|\leq C\limsup_{n\to+\infty}\|u_n-u_\e\|_{H^1(A(\boldsymbol{x}_n, d_n))} +
\limsup_{n\to+\infty}\|u_\e-u\|
\leq C\e.
\]
Since $\e$ is arbitrary and $C$ is independent of $\e$, we have proved that, up to a subsequence, $u_n\to u$ in $X_q$.

Choose $t_{n,i}> 0$ for $i=1,\cdots,k$ such that
$$\tilde u_n= u_n+\sum_{i=1}^{k }t_{n,i} u_i^\delta(\cdot-x_i^n+x_i)\in S_{\boldsymbol x_n,d_n}.$$
Then by
$$\frac{\rd}{\rd t}J_\delta(tu_i^\delta(\cdot-x_i^{n}+x_i))|_{t=1+t_{n,i}}=0\ \ \text{and }\ x_i^{n}\to x_i\ \text{as }n\to\infty.$$
we can prove $|t_{n,i}|\leq C$. Up to a subsequence, $t_{n,i}\to t_i$ as $n\to\infty$. Therefore, $\frac{\rd}{\rd t}J_\delta(tu_i^\delta)|_{t=1+t_i}=0$, which implies $t_i=0$.
Thus, we have proved $t_{n,i}\to 0$ for $i=1,\cdots,k$, and hence
\[ \limsup_{n\to+\infty} \mu_{d_n}(\boldsymbol x_n)\leq   \limsup_{n\to+\infty} I (\tilde u_n)
=I ( u)=\mu_d(\boldsymbol x),
\]
which establishes upper semi-continuity.

To show $\mu_d(\boldsymbol{x})$ is lower semi-continuous, we set $u_n\in M_{\boldsymbol{x_n},d_n}$. By the same proof as that in Proposition \ref{pro2.5},
  \[ 
    \mu_{d_n}(\boldsymbol{x_n}) \le I((h_n)_\delta+ \sum_{i=1}^kt_{w_{x_i}}(w_{x_i^n})^\delta)\leq Ck,\ \ \text{where }h_n:=\max_{i\in\{1,\dots,k\}}\{w_{x_i^n}\}.
  \]
By Lemma \ref{lem2.5}, $\|u_n\|\leq Ck^{\theta'}$.
Assume that $u_n\rightharpoonup u_0$ in $X_q$. Then $u_0\in H_0^1(A(\boldsymbol{x}, d))$  emerges around the points $\boldsymbol{x}=\{x_j\}_{j=1}^k$ with $\beta_i(u_0)=0$, and by Lemma \ref{lem2.3},
\begin{equation*} 
\int_{\R^N}|\nabla (u_0)_i^\delta|^2-K(x)((u_0)_i^\delta)^2dx
\leq \liminf_{n\to\infty}
\int_{\R^N}|\nabla (u_n)_i^\delta|^2-K(x)((u_n)_i^\delta)^2dx\leq -c.
\end{equation*}
So there exist
$t_i>0$, such that $\overline u_0:=(u_0)_\delta+\sum_{i=1}^k t_i (u_0)_i^\delta\in S_{\boldsymbol x,d}$. Then we have
$$\mu_d(\boldsymbol x)\leq I(\overline u_0)
= I((u_0)_\delta+\sum_{i=1}^k t_i (u_0)_i^\delta)
\leq \liminf_{n\to\infty}
I((u_n)_\delta+\sum_{i=1}^k t_i (u_n)_i^\delta)
\leq \liminf_{n\to\infty}I(u_n)
=\liminf_{n\to\infty}\mu_{d_{n}}(\boldsymbol {x_{n}}).$$
\end{proof}

\begin{lemma}\label{lem lambda}
  For $k\in\N\setminus\{0\}$ and $ \boldsymbol x  \in\mathcal K_k $,   let $u\in S_{\boldsymbol x,d}$ be a minimizer such that $I(u)=\mu_d(\boldsymbol x)$. Then, for $i=1,\cdots,k$, there exist $\lambda_{i,d}\in\R^N$ such that
\begin{equation*} 
I'(u)h=\int_{B(x_i,\rho)}u^\delta(x)h(x)(\lambda_{i,d}\cdot(x-x_i))\rd x,\ \ \ \forall h\in C_0^\infty(B(x_i,\rho)).
\end{equation*}
\end{lemma}
\begin{proof}
Assume to the contrary that for some $j\in\{1,\cdots,k\}$ the lemma does not hold. That is for any $\lambda\in\R^N$, there is $h\in C_0^\infty(B(x_j,\rho))$
such that
\begin{equation}\label{eq31}
I'(u)h\neq\int_{B(x_j,\rho)}u^\delta(x)h(x)(\lambda\cdot(x-x_j))\rd x.
\end{equation}
Define $E:=\{v\in H^1(B(x_j,\rho))\ |\ v=u\ \text{on }\partial B(x_j,\rho), \  \beta_j(v)=0\}$.  By \eqref{eq31}, $I'|_{E}(u)\neq 0.$
Then there is $r\in (0,r_1)$ such that $I'|_{E}(v)\neq 0$ for
$v\in E_r:=\{v\in E\ |\ \|v-u\|_{H^1(B(x_j,\rho))}\leq r\}$, where $r_1$ is given in Lemma \ref{belowbound}. According to Lemma \ref{lem2.3}, making
$r$ smaller if necessary, we can assume that
$\int_{B(x_j,\rho)}|\nabla v^\delta|^2-K(x)(v^\delta)^2dx<0$ for $v\in E_r$.

There is  a locally Lipschitz pseudogradient field $\mathcal V:E\to H_0^1(B(x_j,\rho))$ such that $\mathcal V(v)\in T_{v}E:=\{h\in H_0^1(B(x_j,\rho))\ |\ \beta_j'(v)h=0\}$, $I'(v) \mathcal V(v)\geq \|I|_E'(v)\|$, and $\|\mathcal{V}(v)\|_{H^1} \leq 2$ for $v\in E$.
Let $\chi$ be a nonnegative, continuous   cut-off function that equals $1$ in $E_{\frac r4}$ and $0$ outside of $ E_{\frac r2}$.
Consider the following initial value problem $\eta: [0,+\infty)\times E\to E$:
$$
\begin{cases}
\frac{\rd}{\rd s}\eta(s,v)=-\chi(\eta)\mathcal V(\eta);\\
\eta(0,v)=v.
\end{cases}
$$
Set $\gamma(t)=u+tu_j^\delta\in E$. Then for $t_0= \frac 34\cdot r/\| u_j^\delta\|_{H^1(B(x_j,\rho))}$, it holds $\gamma(t)\in E_{r}$ for $|t|\leq t_0$ and $\gamma(\pm t_0)\in E_{r}\setminus E_{\frac r2}$.
By lemma \ref{lem2.3}, there is a positive constant $c>0$ such that $\|u_j^\delta\|\geq c$. So making $r$ smaller if necessary, we have $t_0\in(0,\frac12)$. For $t\in[-t_0,t_0]$, we
 define
 $$\gamma_1(t)(x)=
 \begin{cases}
 \eta(1,\gamma(t))(x),\ \ &x\in B(x_j,\rho),\\
 u(x),\  \ &x\notin B(x_j,\rho),
 \end{cases}\quad 
 \gamma_2(t)=\gamma_1(t)^+.
 $$
Then we have $(\gamma_1(t))^\delta =(\gamma_2(t))^\delta$ and
\begin{enumerate}
    \item  $\gamma_2(t)=\gamma_1(t)=\gamma(t)$ for $|t|= t_0$ in $B(x_j,\rho)$.
    \item  $I(\gamma_2(t))\leq I(\gamma_1(t))<I(u)$ for  $|t|\leq t_0$. This is because of $\|(\gamma_1 (t))^-\|_{2}\leq r<r_1$ and Lemma \ref{belowbound}.
      \item $ \beta_j(\gamma_2(t))= \beta_j(\gamma_1(t))=0$ for $|t|\leq t_0$, which follows from the definition of $\beta_j$.
      \item $\int_{B(x_j,\rho)}|\nabla \gamma_2(t)^\delta|^2-K(x)(\gamma_2(t)^\delta)^2dx<0$ since $\eta(1,\gamma(t))\in E_r$ for 
      $|t|\leq t_0$.
  \end{enumerate}
By (i) and Lemma \ref{lem2.1}, $I'(\gamma_2(t))(\gamma_2(t))_j^\delta<0$ for $t=t_0$, 
and $I'(\gamma_2(t))(\gamma_2(t))_j^\delta>0$ for $t=-t_0$. Hence there is $t_1\in(-t_0,t_0)$ such that 
$I'(\gamma_2(t_1))(\gamma_2(t_1))_j^\delta=0$. This means that
$\gamma_2(t_1)\in S_{\boldsymbol{x},d}$ and consequently
$I(\gamma_2(t_1))\geq \mu_d(\boldsymbol x)= I(u)$, which contradicts to (ii).
\end{proof}

\begin{lemma}\label{lem414}
For $d\geq 0$, there hold  $\mu_{1,d}^\infty=\mu_1^\infty=m_0:=I^\infty(w)$, and $M_{y,d}^\infty=M_y^\infty=\{w_y\ |\ y\in \R^N\}$.
\end{lemma}
\begin{proof}
Since 
for $d\geq 0$, $y\in\R^N$
$w_y\in S_{y,d}^\infty\subset S_y^\infty$, we have
$$m_0=I^\infty(w)\geq  \mu_d^\infty(y)\geq \mu^\infty(y)=\mu_1^\infty.$$
On the other hand, we take $u\in M_0^\infty$. Then
$g_u(t):=(I^\infty)'(u_\delta+tu^\delta)(u_\delta+tu^\delta)$ satisfies
$$
\begin{aligned}
g_u(t)
=t^2\int_{\R^N}|\nabla u^\delta|^2-(u^\delta)^2dx
+\int_{\R^N}(u_\delta+tu^\delta)^{q}  dx
-2\delta t\int_{\R^N} u^\delta dx
+\int_{\R^N} |\nabla u_\delta|^2-(u_\delta)^2 dx,
\end{aligned}
$$
which implies that $g_u(t)\to -\infty$ as $t\to+\infty$ and
$g_u(0)>0$. Then there is $t_0>0$ such that $g_u(t_0)=0$.
So we have $u_\delta+t_0u^\delta\in \mathcal N$, and consequently
$$\mu_1^\infty=\mu^\infty(0)=I^\infty(u)\geq I^\infty(u_\delta+t_0u^\delta)\geq\inf_{\mathcal N} I^\infty(u)=m_0,$$
where $\mathcal N$ is the Nehari manifold for $I^\infty$ given at the beginning of Section 2. Moreover, applying Lemma \ref{lem2.4} with $K=1$, we have $t_0=1$.

According to the above proof, for any $u\in M_{0}^\infty$, it holds $u\in\mathcal N$ and $I^\infty(u)=m_0$. Hence it is a ground state solution for $I^\infty$. By uniqueness,
$u=w$. Consequently, $M_{y,d}^\infty=M_y^\infty=\{w_y\ |\ y\in \R^N\}$.
\end{proof}
 By (K4), for every $K$ there is 
 $a_2(K)>0$ such that $K<1$ outside $B(0, a_2(K))$.
\begin{proposition}\label{muy}
Let $d\geq 0$,  $k\in \mathbb N\setminus\{0\}$. Then the following statements hold:
\begin{enumerate}
    \item For every \( y \in \mathbb{R}^N \setminus B(0, R_0 + a_2(K)) \), we have  
          \( \mu_d(y) \geq \mu(y) > m_0 \).
    \item  
          \( \lim_{|y| \to +\infty} \mu_d(y) = m_0 \).
          \item \( \mu_{1,d} \geq \mu_1 > m_0 \) and $\mu_{k+1,d}>\mu_{k,d}+m_0$.
    \item The supremum \( \mu_{k,d}=\sup_{\boldsymbol y\in\mathcal K_k}\mu_d(\boldsymbol y) \) is attained.
\end{enumerate}
\end{proposition}
 
\begin{proof}
 Let $ u \in M_y $. Then $ \operatorname{supp} u \subset B(y, R_0) $, and by assumption (K4), we have $ K < 1 $ on this ball. From Lemma \ref{lem414}, it follows that
\[
\mu_d(y) \geq \mu(y) = I(u) \geq I(\widehat{u}) > I^\infty(\widehat{u}) \geq \mu_d^\infty(y) = m_0,
\]
where $ \widehat{u} := u_\delta + t u^\delta \in S_y^\infty $ for some $ t > 0 $.

To prove (ii), we consider $\widehat w_y:=(w(\cdot-y))_\delta+t_y(w(\cdot-y))^\delta\in S_{d,y}\subset S_{y}$.
By Lemma \ref{lem 4.8},  \( t_y \) is bounded. Together with the fact that $K(x)\to1$ as $|x|\to\infty$, 
we  obtain 
\[
\limsup_{|y|\to +\infty}\mu_d(y)\leq \limsup_{|y|\to +\infty} I (\widehat w_y)\leq\limsup_{|y|\to +\infty}I^\infty(\widehat w_y)+\limsup_{|y|\to +\infty}\int_{\R^N}(1-K(x))\widehat w_y^2\leq \limsup_{|y|\to +\infty}I^\infty(w) =m_0.
\]Since \( \mu_d(y) > m_0 \) for large \( y \) by (i), we conclude that \( \lim_{|y| \to +\infty} \mu_d(y) = m_0 \). 

We prove (iii) and (iv) by induction. Consider the case that $k=1$. By (i), we have $\mu_{1,d}\geq\mu_1>m_0$. To show $\mu_{1,d}$ can be attained, we consider a sequence $\{x_n\}\subset\R^N$ such that $\mu_d(x_n)\to\mu_{1,d}$ as $n\to\infty$.
By the continuity of $\mu_d$ given in Lemma \ref{lem mu}, it suffices to 
prove that $|x_n|$ is bounded. Otherwise, we assume without loss of generality that
$|x_n|\to +\infty$. Then by (ii), we have $\mu_{1,d}=\lim_{n\to+\infty}\mu_d(x_n)=m_0$,
which contradicts to (i). Let $y\in  \R^N$ be such that $\mu_{1,d}=\mu_d(y)$.
By (K4), choose
$z\in\R^N$ with $|z|$ sufficiently large such that $K(x)<1$ for $x\in B(z,R_0)$ and $\min_{1\leq i\leq k}|y-z|\geq 2R_0$.
Then by Corollary \ref{cor413} and  (i),
\begin{equation*}
	\mu_{2,d}\geq \mu_d(\{y,z\})= \mu_d(y)+\mu_d(z)> \mu_{1,d}+m_0.
\end{equation*} 
Thus we have proved that (iii) and (iv) hold for $k=1$.

To complete the proof, we assume that (iii) and (iv) hold for some $k\geq 1$. 
Let $\boldsymbol{y}=\{y_i\}_{i=1}^k\in  \mathcal{K}_k$
be such that $\mu_{k,d}=\mu_d(\boldsymbol{y})$.
By (K4), choose
$y_{k+1}\in\R^N$ with $|y_{k+1}|$
sufficiently large such that $K(x)<1$ for $x\in B(y_{k+1},R_0)$ and $\min_{1\leq i\leq k}|y_{k+1}-y_i|\geq 2R_0$.
Then by Corollary \ref{cor413} and  (i),
\begin{equation*}
\mu_{k+1,d}\geq \mu_d(\boldsymbol{y}\cup \{y_{k+1}\})= \mu_d(\boldsymbol{y})+\mu_d(y_{k+1})> \mu_d(\boldsymbol{y})+m_0=\mu_{k,d}+m_0.
\end{equation*} 
Consider a sequence $\boldsymbol x_n=\{x_i^n \}_{i=1}^{k+1}\in \mathcal K_{k+1}$ such that
$\mu_d(\boldsymbol x_n)\to\mu_{k+1,d}$ as $n\to\infty$.
By  the  continuity of $\mu_d$ given in Lemma \ref{lem mu}, to prove (iv) holds for $k+1$, it suffices to 
show that $\max_{1\leq i\leq k+1}|x_i^n|$ is bounded. Otherwise, we assume without loss of generality that
$|x_1^n|\to +\infty$.
Then by Corollary \ref{cor413} and (ii), we have 
\[
\mu_{k+1,d}=\lim_{n\to+\infty}\mu_d(\boldsymbol x_n)\leq
\limsup_{n\to+\infty}\mu_d(\{x_j^n\}_{j\neq 1}) +\lim_{n\to+\infty}\mu_d(x_1^n)\leq \mu_{k,d}+m_0,
\]
which contradicts to $\mu_{k+1,d}>\mu_{k,d}+m_0$.
\end{proof}

\section{Existence of solutions}
This section is dedicated to completing the proof of Theorem \ref{them1.1} on the existence of infinitely many multi-bump solutions. The overarching strategy is to analyze the behavior of the minimizers constructed in Section 4 as the potential 
$K$ approaches the $1$. To this end, we consider a sequence of potentials $(K_n)_n$
such that $\|K_n-1\|_{p,loc}\to 0$. For each $K_n$ we define the corresponding energy functional 
 $$I_n(u) =\frac12\int_{\RN}|\nabla u|^2-K_n(x) u^2 dx
 +\int_{\RN}  \frac1q|u|^q.$$
 Consequently, all related concepts, such as the sets of constrained functions 
$S^n_{\boldsymbol x,d}$ and the sets of minimizers $M^n_{\boldsymbol x,d}$, are defined in the same way as their counterparts in Section 4, but with respect to the functional $I_n$. Our goal is to show that for large 
$n$, any minimizer $u_n\in M^n_{\boldsymbol x,d}$
not only solves the localized constrained problem but indeed satisfies the original equation on the entire $\R^N$.
 This is achieved by proving that the Lagrange multipliers associated with the barycenter constraints vanish for large $n$, which relies crucially on the qualitative stability estimates from Section 3 and the precise control of the support of $u_n$.

\subsection{Behavior of $M_{\boldsymbol x}$ as $K(x)\to 1$}
\begin{lemma}\label{lem regu}
  Let $k \in \mathbb{N}\setminus\{0\}$, $ \boldsymbol x =\{x_i\}_{i=1}^k \in\mathcal K_k$, $d\in(0,\sigma_0]$, and $R\in (\rho,R^*]$.
Assume $u\in M_{\boldsymbol{x}, d}$ and set $v=\sum_{i=1}^k w(\cdot-x_i)$. 
Then there exists a constant $C_R > 0$, depending only on 
  $R$, but independent of $k$, $\boldsymbol{x}$, and $d$, 
 \[
 \|u-v\|_{L^\infty(A(\boldsymbol{x}, d) \setminus \bigcup_{i=1}^{k}\overline {B(x_i,R)})}\leq
 C_R  \left(\max_{1\leq i \leq k}  \|u-v\|_{L^2(B(x_i, R^*+d))}+\sigma(\boldsymbol{x}) +\|K-1\|_{p,loc}\right),\]
\end{lemma}
\begin{proof}
According to Lemma \ref{lem2.10}, the following equation weakly holds in $A(\boldsymbol{x}, d) \setminus \bigcup_{i=1}^{k}\overline {B(x_i,\rho)}$:
  \[
  -\Delta(u- v) -K(u-v) +  u^{q-1}- v^{q-1} = (K-1)v +(\sum_{i=1}^k w(\cdot-x_i)^{q-1}-v^{q-1}).
  \]
  Setting $V=(u^{q-1}-v^{q-1})/(u-v)$ and $h=(K-1)v +(\sum_{i=1}^k w(\cdot-x_i)^{q-1}-v^{q-1})$, then we have the following equation  
 \[
  -\Delta(u- v)+(V-K)(u-v) = h\ \ \text{in }A(\boldsymbol{x}, d) \setminus \bigcup_{i=1}^{k}\overline {B(x_i,\rho)}.
  \]
 Since $u,v\in[0,\delta)$ outside $\bigcup_{i=1}^k\overline {B(x_i,\rho)}$ and $K\leq a_1$, we have $V-K>0$ in $A(\boldsymbol{x}, d) \setminus \bigcup_{i=1}^{k}\overline {B(x_i,\rho)}$. 
Noting that $u, v=0$ outside $A(\boldsymbol{x}, d)$ and
$R^*+d\leq R^*+\sigma_0<R_0$, by standard Moser's iteration, we have for each $i\in\{1,\cdots,k\}$ that
   \[
 \|u-v\|_{L^\infty(B(x_i,R^*+d) \setminus\overline {B(x_i,R)})}\leq
 C  \left( \|u-v\|_{L^2(B(x_i, R_0))}+\|h\|_{L^{p'}(B(x_i,R_0))}\right),\]
 for $p'=\min\{p, (N+1)/{2}\}>N/2$.
 By \eqref{eq 9}, $R^*+\sigma_0<R^*+\sigma_0^{\frac{2-q}{2}}=R_0$, $\supp h\subset \A 0$, and Lemma \ref{lem 2.3}, 
 \begin{equation*} 
 \|h\|_{L^{p'}(B(x_i,R_0))} \leq \sum_{j\in \mathscr{I}_{i}}\|h\|_{L^{p'}(B(x_j,R^*))} \leq Ck_0(\sigma(\boldsymbol{x})+\|K(x)-1\|_{p,loc}).
 \end{equation*}
 And similarly, 
  \begin{equation*} 
 \|u-v\|_{L^2(B(x_i, R_0))}\leq \sum_{j\in \mathscr{I}_{i}}\|u-v\|_{L^2(B(x_j,R^*+d))}\leq k_0\max_{1\leq j \leq k}  \|u-v\|_{L^2(B(x_j, R^*+d))}.
 \end{equation*}
Then the conclusion follows from the abitrary choice of $i\in\{1,\cdots,k\}$.
\end{proof}

\begin{proposition}\label{pro3.1} 
Let $k_n \in \mathbb{N}\setminus\{0\}$, $\boldsymbol{x}_n = \{x_i^n\}_{i=1}^{k_n} \in \mathcal{K}_{k_n}$, and $d_n \geq 0$ be such that $\sigma(\boldsymbol{x}_n)\to 0$, $\|K_n -1\|_{p,loc}\to 0$,  $ d_n \to 0$.
 Then there is $C>0$ indendent of $n$ such that  for any $u_n\in M^n_{\boldsymbol x_n,d_n}$,  it holds that 
 \begin{enumerate}
 	\item \(  \|u_n\|_{\boldsymbol{x}_n, d_n}\leq C \).
 	\item \( \lim_{n\to+\infty}    \max_{1\leq j\leq k_n} \|u_n - w(\cdot-x_j^n)\|_{L^2(B(x_j^n, R^*))}=0 \).
 	\item  
 $\supp(u_n)^\delta\subset \bigcup_{i=1}^{k_n}B(x_i^n,\rho-\sigma_0/2)$ for large $n$.
 \end{enumerate}
\end{proposition}
\begin{proof}
Recalling \eqref{26} that $u_n$ satisfies
 \begin{equation*}  
        -\Delta u_n-K_n(x)u_n+u_n^{q-1}=0,\quad\ 0\leq   u_n < \delta\ \text{in\ }A(\boldsymbol{x_n}, d_n) \setminus \bigcup_{j=1}^{k_n}\overline {B(x_j^n,\rho)}.
\end{equation*}

\noindent\textbf{Step 1.} For any $j_n\in\{1,\cdots,k_n\}$ , there is $\bar v\in H^1_0(B(0,R^*))$ such that
\begin{equation}\label{eq 45}
u_n(\cdot+x_{j_n}^n)\rightharpoonup \bar v\ \text{in }
H^1\big(B(0,R^*)\setminus \overline{B(0,\rho+\sigma_0/2)}\big).
\end{equation}
Let
$\phi\in C_0^\infty(B(0,R^*+2\sigma_0)\setminus \overline{B(0,\rho)},[0,1])$ be a cut-off function such that  
$\phi=1$ in $B(0,R^*+\sigma_0)\setminus B(0,\rho+\sigma_0/2)$  and $|\nabla\phi|\leq C$. 
Then by $\rho= R^*-3\sigma_0$, $|x_i^n-x_j^n|\geq 2R^*-\sigma(\boldsymbol{x}_n)=2R^*+o_n(1)$ for $i\neq j$, we know that $\supp \phi(\cdot-x_{j_n}^n)\subset B(x_{j_n}^n,R^*+2\sigma_0)\setminus \bigcup_{j=1}^{k_n}\overline {B(x_j^n,\rho)}$. Setting $\phi_n=\phi(\cdot-x_{j_n}^n)$, we have 
$$\phi_nu_n,\ \phi_n^2u_n\in H_0^1\big(A(\boldsymbol{x_n}, d_n) \setminus \bigcup_{j=1}^{k_n}\overline {B(x_j^n,\rho)}\big).$$
So we can test
the equation with $\phi_n^2 u_n$, and we get
$$\int_{\R^N} |\nabla (\phi_nu_n)|^2+\phi_n^2(\delta^{q-2}-a_1)u_n^2  \leq \int_{\R^N}|\nabla\phi_n|^2u_n^2 \leq C.$$
This, together with the fact $\phi_n\in C_0^\infty(B(x_{j_n}^n,R^*+2\sigma_0))$, implies the boundedness of $\{\phi_nu_n\}$ in $H^1(\R^N)$. In particular, 
\begin{equation}\label{eq5.2}
\|u_n\|_{H^1(B(x_{j_n}^n, R^*+\sigma_0)\setminus \overline{B(x_{j_n}^n,\rho+\sigma_0/2)})}\leq C.	
\end{equation}
Moreover, $\phi_nu_n=0$ in $B(x_{j_n}^n,R^*+2\sigma_0)\setminus \cup_{j\in\mathscr I_{j_n}^n }B(x_j^n, R^*+d_n)$ with \(\mathscr{I}_{j_n}^n = \{ j \mid x_j^n \in B(x_{j_n}^n, 2R_0) \}\). According to Lemma \ref{lem2.2},  
$\phi u_n(\cdot+x_{j_n}^n) \rightharpoonup  v$ in $H^1(B(0, R^*+2\sigma_0))$, and $\bar v:=v|_{B(0, R^*)} \in H_0^1(B(0, R^*)$. Recalling that $\phi=1$ in $B(0,R^*)\setminus B(0,\rho+\sigma_0/2)$, we have \eqref{eq 45}.

\noindent\textbf{Step 2.}
Let
$$v_{n}:=(1-\varphi(\cdot-x_{j_n}^n)) u_n
+  (\varphi \bar v)(\cdot-x_{j_n}^n) \in H_0^1(B(x_{j_n}^n,R^*)),$$
 with $\varphi\in C_0^\infty(\R^N\setminus \overline{B(0,\rho+\sigma_0/2)},[0,1])$ being  a cut-off function such that $\varphi=1$ in $\R^N\setminus B(0,\rho+\sigma_0)$.
 We prove that  
\begin{equation}\label{eq 46}
I_n (u_n) 
\geq I_n (v_{n})+ I_n(u_n-v_n)+o_n(1).
\end{equation} 
In fact, we have 
$$\int_{\R^N}|\nabla u_n|^2=\int_{\R^N}|\nabla (u_n-v_n)|^2+|\nabla v_n|^2+2\nabla (u_n-v_n)\nabla v_n,$$
and
\begin{equation*}
\begin{aligned}
\int_{\R^N}\nabla(u_n-v_n)\nabla v_n
=\int _{B(0,R^*)}\nabla\Big(\varphi\big(u_n(\cdot+x_{j_n}^n)-\bar v\big)\Big) \nabla \Big((1-\varphi) u_n(\cdot+x_{j_n}^n)
+  \varphi \bar v\Big).\\
\end{aligned}
\end{equation*}
Since $u_n(\cdot+x_{j_n}^n)\rightharpoonup \bar v\ \text{in }
H^1(B(0,R^*)\setminus B(0,\rho+\sigma_0/2))$, we have 
\begin{equation*}
\begin{aligned}
\int_{\R^N}\nabla(u_n-v_n)\nabla v_n
&=\int _{B(0,R^*)\setminus B(0,\rho+\sigma_0/2)}\varphi(1-\varphi)\nabla\big(u_n(\cdot+x_{j_n}^n)-\bar v\big) \nabla  u_n(\cdot+x_{j_n}^n)+o_n(1)\\
&=\int _{B(0,R^*)\setminus B(0,\rho+\sigma_0/2)}\varphi(1-\varphi)\big(|\nabla u_n(\cdot+x_{j_n}^n)|^2-|\nabla \bar v|^2\big) +o_n(1)\\
&\geq o_n(1),
\end{aligned}
\end{equation*}
where the last inequality follows from the facts that
$u\to\int_{B(0,R^*)\setminus B(0,\rho+\sigma_0/2)}\varphi(1-\varphi)|\nabla u|^2$ is continuous and convex on $H^1(B(0,R^*)\setminus B(0,\rho+\sigma_0/2))$, and hence is weakly lower semi-continuous. Thus we have proved that
$$\int_{\R^N}|\nabla u_n|^2\geq\int_{\R^N}|\nabla (u_n-v_n)|^2+|\nabla v_n|^2+o_n(1).$$
Moreover, by  \eqref{eq 45}, we have
$ 
u_n(\cdot+x_{j_n}^n)\to \bar v\ \text{in }
L^q(B(0,R^*)\setminus B(0,\rho+\sigma_0/2)),
$
and hence
$$\|u_n-v_n\|_{L^q(B(x_{j_n}^n,R^*))}=\|\varphi(u_n(\cdot+x_{j_n}^n)-v)\|_{L^q(B(0,R^*))}=o_n(1).$$
Combining this with $v_n=0$ outside $B(x_{j_n}^n,R^*)$, we deduce that
$$
\begin{aligned}
\int_{\R^N}|u_n|^q
&=\int_{B(x_{j_n}^n,R^*)}|u_n|^q+\int_{\R^N\setminus B(x_{j_n}^n,R^*)}|u_n|^q\\
&=\int_{B(x_{j_n}^n,R^*)}|u_n-v_n|^q+|v_n|^q+\int_{\R^N\setminus B(x_{j_n}^n,R^*)}|u_n-v_n|^q+|v_n|^q
+o_n(1).
\end{aligned}$$
Similarly,
$$
\begin{aligned}
\int_{\R^N}K_n(x)|u_n|^2
=\int_{\R^N}K_n(x)|u_n-v_n|^2+K_n(x)|v_n|^2+o_n(1).
\end{aligned}$$
By these,  \eqref{eq 46} is true.

\noindent\textbf{Step 3.} $\bar v=w$ in $B(0,R^*)\setminus B(0,\rho+\sigma_0/2)$, and it holds
 \begin{equation*} 
 u_n(\cdot+x_{j_n}^n)\rightharpoonup w\ \ \text{in }
H^1(B(0,R^*)\setminus \overline{B(0,\rho+\sigma_0/2)});\ \ \ 
u_n(\cdot+x_{j_n}^n)\to w\ \ \text{in }
H^1(B(0,\rho+\sigma_0/2)).
\end{equation*} 
Let $t_n>0$ be such that
 $w_{n}:=(w(\cdot-x_{j_n}^n))_\delta+t_{n}(w(\cdot-x_{j_n}^n))^\delta\in S^n_{x_{j_n}^n,d_n}$. Then by $u_n-v_n=0$ in $B(x_{j_n}^n,\rho+\sigma_0/2)$, $u_n-v_n=u_n$ in
 $A(\boldsymbol{x_n}, d_n) \setminus  B(x_{j_n}^n,R^*)$,
 $|u_n-v_n|=|\varphi(\cdot-x_{j_n}^n) (u_n-\bar v)|\in[0,\delta]$ in $B(x_{j_n}^n,R^*)\setminus B(x_{j_n}^n,\rho+\sigma_0/2)$,
  $u_n\in M_{\boldsymbol x_n,d_n}^n$,   and 
 $\sigma(\boldsymbol{x}_n)\to 0$, it holds 
$w_{n}\vee (u_n-v_n)\in S_{\boldsymbol x_n, d_n}^n$.
  So we have
\begin{equation}\label{37}
I_n (u_n)\leq I_n(w_{n}\vee (u_n-v_n))= I _n(w_{n})+I_n(u_n-v_n)-I_n(w_{n}\wedge (u_n-v_n))
\leq m_0+ I_n(u_n-v_n)+o_n(1),
\end{equation}
where $I_n(w_n)\leq m_0+o_n(1)$ since $\{t_n\}$ are bounded and $I_n(w_n)=I^\infty(w_n)+o_n(1)\leq I^\infty(w)+o_n(1)$,
$I_n(w_{n}\wedge (u_n-v_n))\geq 0$ since $|w_{n}\wedge (u_n-v_n)|\leq \delta$. 
By \eqref{eq 46} and \eqref{37},
\begin{equation} \label{eq 48}
I_n (v_{n}) 
\leq m_0 +o_n(1).
\end{equation}
Recalling the definition of $v_n\in H_0^1(B(x_{j_n}^n,R^*))$, we know that $v_n=u_n$ in $B(x_{j_n}^n,\rho+\sigma_0/2)$ and
$v_n\in[0,\delta]$ in $B(x_{j_n}^n,R^*)\setminus \overline{B(x_{j_n}^n,\rho+\sigma_0/2)}$.
Thus  $v_{n}\in  S_{x_{j_n}^n  }^n$, combining which with Lemma \ref{lem2.5},
we obtain that $\{v_n\}$ are bounded in $X_q$.
Let $\tilde v_{n}= v_{n} +s_{n} (v_n)^\delta\in S^\infty_{0}$. Then  by 
$$
\begin{aligned}
0&=\frac{d}{dt}I^\infty((v_n)_\delta+t(v_n)^\delta)|_{t=1+s_n}\\
&=(1+s_n)\int_{\R^N}|\nabla (v_n)^\delta|^2-((v_n)^\delta)^2dx
+\int_{\R^N}(\delta+(1+s_n)(v_n)^\delta)^{q-1}(v_n)^\delta dx
-\delta\int_{\R^N} (v_n)^\delta dx,
\end{aligned}
$$
Lemma \ref{lem2.3}, and the bounded-ness of $\{v_n\}$ in $X_q$, we have $|s_n|\leq C$. Then by Lemma \ref{lem2.4},
$$m_0\leq I^\infty(\tilde v_n)=I_n(\tilde v_n)+o_n(1)
\leq I_n(v_n)-c_0\min\set{(s_n)^2,1}+o_n(1).$$
This, together with \eqref{eq 48}, implies that $s_n\to 0$ as $n\to\infty$.
Thus 
  $ \tilde v_{n} (\cdot + x_{j_n}^n) \in S_{0}^\infty$ is a minimizing sequence for $\mu^\infty(0)$ since
  $I^\infty(\tilde v_n)=I^\infty(v_n)+o_n(1)\leq m_0+o_n(1)$.
 Applying Proposition \ref{pro2.5} with $K\equiv1$ and Lemma  \ref{lem414}, we obtain $\tilde v_{n} (\cdot + x_{j_n}^n) \to w$ in $X_q$. Since $s_n\to 0$, it holds $v_n(\cdot + x_{j_n}^n)\to w$ in $X_q$. On the other hand, by \eqref{eq 46} and the definition of $v_n$, we have $v_n(\cdot+x_{j_n}^n)\rightharpoonup v\ \text{in }
H^1(B(0,R^*)\setminus \overline{B(0,\rho+\sigma_0/2)})$, and hence
$v=w$ in $B(0,R^*)\setminus \overline{B(0,\rho+\sigma_0/2)}$. Taking this into \eqref{eq 45}, we get
 \begin{equation*} 
u_n(\cdot+x_{j_n}^n)\rightharpoonup w\ \ \text{in }
H^1(B(0,R^*)\setminus \overline{B(0,\rho+\sigma_0/2)}).
\end{equation*}
In $B(0,\rho+\sigma_0/2)$, $v_n(\cdot + x_{j_n}^n)=u_n(\cdot + x_{j_n}^n)$, and hence
 \begin{equation*} 
u_n(\cdot+x_{j_n}^n)\to w\ \ \text{in }
H^1(B(0,\rho+\sigma_0/2)).
\end{equation*} 
By these, \eqref{eq 9} and \eqref{eq5.2}, we have (i) and (ii). 
 
\noindent\textbf{Step 4.} It holds $\supp(u_n)^\delta\subset \bigcup_{i=1}^{k_n}B(x_i^n,\rho-\sigma_0/2)$ for large $n$.

Let $i_n\in\{1,\cdots,k_n\}$ be arbitrary.
According to Step 1--3, $u_n(\cdot+x_{i_n}^n)\to w$ in 
$H^1(B(0,\rho+\sigma_0/2))$. Moreover,
by Lemma \ref{lem2-3} and  Lemma \ref{lem lambda}, we know that for some $\lambda_n\to 0$,
\begin{equation} \label{eq 49}
-\Delta u_n  -u_n  + u_n^{q-1}  =  (K_n-1) u_n +\lambda_{n} \cdot(x-x_{i_n}^n) (u_n)_{i_n}^\delta\  \text{  in } H^1(B(x_{i_n}^n, \rho)).
\end{equation} 
 By elliptic regularity theory, $u_n(\cdot+x_{i_n}^n)$ is bounded in $C^{\beta}(U)$ for some $\beta\in(0,2-n/p)$ in any compact subset $U\subset B(0,\rho)$.
Therefore, $u_n(\cdot+x_{i_n}^n)$ converges uniformly to $w$ in any compact subset $U\subset B(0,\rho)$.
 Noting that
$w(x)<\delta$ for $|x|> \rho-\sigma_0$, up to a subsequence if necessary, we have
\begin{equation*} 
u_{n}<\delta\ \ \text{for }x\in B(x_{i_n}^{n},\rho-\sigma_n)\setminus B(x_{i_n}^{n},\rho-\frac12\sigma_0),\ \ \text{where }\sigma_n\in(0,\frac12\sigma_0),\ \sigma_n\to 0\ \text{as }n\to\infty.
\end{equation*}
Then we can write 
\[
(u_n)_{i_n}^\delta=u_{n,1} +u_{n,2}, \text{ with } \supp u_{n,1}\subset B(x_{i_n}^n,\rho-\sigma_0/2),\ \ \supp u_{n,2} \subset B(x_{i_n}^n, \rho)\setminus B(x_{i_n}^n, \rho-\sigma_n).
\]
Next we prove $u_{n,2}=0$ for large $n$.
  Testing \eqref{eq 49} by $u_{n,2}$, since $u_{n,1}$ and $u_{n,2}$ has disjoint
support, we get
\begin{equation*}
\begin{aligned}
\int|\nabla u_{n,2}|^2
-K_n(\delta+u_{n,2}) u_{n,2} 
+\int (\delta+u_{n,2})^{q-1} u_{n,2} 
=\int(u_{n,2})^2( \lambda_n \cdot (x-x_{i_n}^n)). 
\end{aligned}
\end{equation*}
 By this and the same proof as \eqref{eq10}, since $\lambda_n\to 0$,  we obtain that
\begin{equation*} 
	\begin{aligned}
		\int_{\R^N}|\nabla u_{n,2}|^2-(K_n(x)+o_n(1))(u_{n,2})^2dx
		&\leq -(1-t_*)\int_{\R^N}(u_{n,2})^{q}dx,
	\end{aligned}
\end{equation*}
where $t_*\in(0,2-q)$ is given in \eqref{eq9}.
Then by \eqref{GN} and the Young inequality, we get
$$
\begin{aligned}
	\|u_{n,2}\|_{q}^q\leq C_{q,N}\|u_{n,2}\|_{q}^2,
\end{aligned}
$$
which implies that $\|u_{n,2}\|_{q}\geq c>0$ if $u_{n,2}\neq 0$.
Noting the following fact
\begin{equation*}
\begin{aligned}
\int|u_{n,2}|^q
\leq |\supp (u_{n,2})|^{\frac{2-q}{2}}\Big(\int|u_{n,2}|^2\Big)^{q/2}
= o_n(1), 
\end{aligned}
\end{equation*} 
we get
 $u_{n,2}=0$ for large $n$. Consequently,  $\supp (u_n)_{i_n}^\delta \subset B(x_{i_n}^n,\rho-\frac12\sigma_0)$ for large $n$.  
\end{proof}
In the following proof, let $\alpha_1 > 0$ be the constant given in Lemma \ref{sss}, and let $\varepsilon_1 > 0$ and $\sigma_1 > 0$ be the constants provided by Proposition \ref{prop3.1}. Then, by combining Lemma \ref{lem2-3}, Lemma \ref{lem2.10}, Lemma \ref{lem lambda}, and Proposition \ref{pro3.1}, we obtain the following corollary.

\begin{corollary}\label{cor3.2}
There exist small $d_0\in(0,\sigma_0)$, $\sigma_2\in(0, \sigma_1)$, and $\alpha_2\in(0,\alpha_1)$, such that  for any 
$k\in\N\setminus\{0\}$,  $\boldsymbol x =\{x_i\}_{i=1}^k\in\mathcal K_{k}$ satisfying $\sigma(\boldsymbol{x})\leq \sigma_2$, any $K$ with  $\|K-1\|_{p,loc}\leq \alpha_2$, 
and  $u\in M_{\boldsymbol x, d_0}$,   there holds 
\[ 
 \max_{1\leq j\leq k} \|u-w(\cdot-x_j )\|_{L^2(B(x_j, R^*))}
  \leq \e_1,
\]
and
$\supp u_i^\delta \subset B(x_i, \rho-\sigma_0/2)$ for each $i=1,\cdots,k$. Moreover, 
 there exist $\lambda_{i,d_0}\in \R^N$ with $|\lambda_{i,d_0}|\leq 1$ for each $i=1,\cdots,k$, such that 
\begin{equation*}
  -\Delta u-K(x)u+ u^{q-1} 
=\sum_{i=1}^{k}\lambda_{i,d_0}\cdot(x-x_i)u_i^\delta\ \ \text{in }\A {d_0}.
\end{equation*} 
In particular, $\|I'(u)\|_{*,u}=0$.
\end{corollary}

For the fixed $d_0$, if we  appropriately reduce $\sigma(\boldsymbol{x})$ and $\|K - 1\|_{p, loc}$,  we will obtain the following proposition.

\begin{proposition}\label{cor5.4}
There exist  $\sigma_3\in(0,\sigma_2)$ and $\alpha_3\in(0,\alpha_2)$,   such that
 for any $k \in \mathbb{N} \setminus \{0\}$, and any   $\boldsymbol{x} = \{x_i\}_{i=1}^k \in \mathcal{K}_k$ satisfying $\sigma(\boldsymbol{x}) \leq \sigma_3$, any $K$ with $\|K - 1\|_{p, loc} \leq \alpha_3$,  and any $u \in M_{\boldsymbol{x}, d_0}$, there holds,
\[
u = 0 \quad \text{in } \A {d_0} \setminus \A{d_0/2}, 
\]
In particular,
\begin{equation}\label{eq5.7}
-\Delta u - K(x)u + u^{q-1} = \sum_{i=1}^k \lambda_i \cdot (x - x_i) u_i^\delta \quad \text{in } \mathbb{R}^N,
\end{equation}
for some $\lambda_i \in \mathbb{R}^N$ with $|\lambda_i|\leq 1$, $i = 1, \cdots, k$.  
\end{proposition}
\begin{proof}   Set 
  \[
  v=\sum_{j= 1}^k  w(\cdot-x_j).
  \]
According to Proposition \ref{pro3.1}, Corollary \ref{cor3.2}, Proposition \ref{prop3.1}, \eqref{K-1} and \eqref{Iinfw},
for some constant $C>0$, it holds 
\begin{equation}\label{equ-v}
   \begin{aligned}   \|u-v\|_{\boldsymbol x, d_0}
   	\leq & C \|(I^\infty)'(u)-(I^\infty)'(v)\|_{*,u} \\
   	\leq & C  \|I'(u)\|_{*,u}+\|(I^\infty)'(u)-I'(u)\|_{*,u}+ \|(I^\infty)'(v)\|_{*,u}\\
    \leq  & C \|K - 1\|_{p, loc} + C\sigma(\boldsymbol{x}).
  \end{aligned} 
\end{equation}
Then Lemma \ref{lem regu}    implies 
 \[ \|u-v\|_{L^\infty(\A{d_0}\setminus \A 0)}\leq  C(  \|K - 1\|_{p, loc} +  \sigma(\boldsymbol{x})).
 \]
According to   Lemma \ref{R0}, $\supp u\subset \A {C(  \|K - 1\|_{p, loc} +  \sigma(\boldsymbol{x}))^\frac{2-q}{3}}$.
Then  we can  choose $\alpha_3,\sigma_3$ smaller if necessary to obtain the conclusion.
\end{proof}
 
\begin{remark}\label{rk5.4}
Suppose that
   $\boldsymbol{x}_n=\{x_i^n\}_{i=1}^{k_n}\in \mathcal K_{k_n}$  satisfy $\sigma(\boldsymbol{x}_n)\to 0$, $K_n$ satisfy $\|K_n-1\|_{p,loc}\to 0$. Then from \eqref{equ-v}, there holds 
   $\|u_n-\sum_{i=1}^{k_n} w(\cdot-x_i^n)\|_{H^1_{loc}}\to 0$ for $u_n\in M_{\boldsymbol{x}_n, d_0}$.
 By \eqref{eq5.7} and
 standard elliptic regularity theory,  we get  $\|u_n \|_{C^\beta(\R^N)}\leq C$,  and hence, by the Arzel\`a--Ascoli theorem and Lemma \ref{lem regu}, we have
  \[ 
   \|u_n-\sum_{i=1}^{k_n} w(\cdot-x_i^n)\|_{L^\infty(\R^N)}+\|u_n-\sum_{i=1}^{k_n} w(\cdot-x_i^n)\|_{H^1_{loc}}\to 0.
  \]
\end{remark}

\subsection{Proof of existence of infinitely many solutions }
The following Proposition states that if $\mu_{d}(\boldsymbol{x})=\mu_{k,d}$, then $\sigma(\boldsymbol{x})\to 0$  as $\|K(x)-1\|_{p,loc}\to0$.
\begin{proposition}\label{pro 3.1}
  For any $\sigma >0$, there is $\alpha_\sigma>0$ such that if
   $\|K(x)-1\|_{p,loc}\leq \alpha_\sigma$, 
  then for any $d\geq 0$, $k\geq 1$,   any $\boldsymbol{x}\in \mathcal K_k$ satisfying $\mu_{d}(\boldsymbol{x})=\mu_{k,d}$ 
 obeys
 $\sigma(\boldsymbol{x})\leq \sigma$.
\end{proposition}
\begin{proof} The case $k=1$ is trivial since $\sigma(\boldsymbol{x})=0$  by definition. Thus, we may assume  $k\geq 2$.
Suppose,  for contradiction, that $\|K_n(x)-1\|_{p,loc}\to 0$, and there exist  
$\boldsymbol x_n=\{x_{i}^n\}_{i=1}^k\in\mathcal K_{k_n}$ satisfying
$\mu_{d_n}(\boldsymbol x_n )=\mu_{k_n,d_n} $ but  
$|x_1^n-x_2^n|\leq 2R^*-\sigma$ for some constant $\sigma>0$. 
Then there are $t_{n,i}>0$ such that
$w_{n,i}=(w_{x_i^n})_\delta+t_{n,i}(w_{x_i^n})^\delta\in S_{x_i^n}^n$, $i=1,2$. 
Set $w_n:=w_{n,1}\vee w_{n,2}$,
and take $v_n\in M_{x_3^n,\dots,x_{k_n}^n,d_n}^n$.
Easily to see, $w_n\vee v_n\in S_{x_1^n,\dots,x_{k_n}^n,d_n}^n$, and $t_{n,i}\to 1$. Hence we have
\begin{equation}\label{eq 17}
\begin{aligned}
\mu_{k_n,d_n}
\leq I_n(w_n\vee v_n)
&\leq I_n(w_n)+I_n(v_n)\\
&\leq I_n(w_{n,1})+I_n(w_{n,2})
-I_n(w_{n,1}\wedge w_{n,2})
+\mu_{k_n-2,d_n}\\
&= 2m_0+\mu_{k_n-2}-I_n(w_{n,1}\wedge w_{n,2})+o_n(1).
\end{aligned}
\end{equation} 
Since $|x_1^n-x_2^n|\leq 2R^*-\sigma$, we have
$I_n(w_{n,1}\wedge w_{n,2})\geq 2c_0>0$. Taking this into \eqref{eq 17}, we deduce that, for large $n$,
$$\mu_{k_n}
\leq 2m_0+\mu_{k_n-2}-c_0,$$
which contradicts to Proposition \ref{muy}.
\end{proof}

\begin{proof}[Proof of Theorem 1]
By Proposition \ref{pro 3.1}, for $d=d_0$ given in Proposition \ref{cor5.4}, 
if $\|K-1\|_{p,loc}\leq \alpha_0:=\min\{ \alpha_3, \alpha_{\sigma_3}\}$, then for 
any $u\in M_{\boldsymbol{x}, d}$ with $\mu_d(\boldsymbol{x})=\mu_{k,d}$, we have $\sigma(\boldsymbol{x})\leq \sigma_3$. So by Proposition \ref{cor5.4}, $u$ satisfies
\[
-\Delta u-K(x)u+ u^{q-1}
=\sum_{i=1}^{k}\lambda_i\cdot(x-x_i)u_i^\delta\ \ \text{in }\R^N.
\]
We show that $\lambda_i=0$ when $\|K-1\|_{p,loc}$ is small enough.
Assume by contradiction that there exist $\boldsymbol{x}_n=\{x_i^n\}_{i=1}^{k_n}$, $K_n$, $u_{n}\in M^n_{\boldsymbol x_n, d}$, and 
$\lambda_i^n\in\R^N$, $i=1,\cdots, k_n$ such that 
$\|K_n-1\|_{p,loc}\to 0, I (u_n)=\mu^n_{k_n,d},$
and 
\begin{equation}\label{eq55}
 -\Delta u_n-K_n u_n+u_n^{q-1}
=\sum_{i=1}^{k_n}\lambda_i^n\cdot(x-x_i^n)(u_n)_i^\delta  \quad  \text{in }\R^N,
\end{equation}
 but   $\max_{1\leq i\leq k_n}|\lambda_i^n|\neq 0$. 
Assume without loss of generality that $$|\lambda_1^n|=\max_{1\leq i\leq k_n}|\lambda_i^n|.$$
By Lemma \ref{lem2-3} and Remark \ref{rk5.4}, we know 
\[
|\lambda_1^n|\to 0.
\]
Define $\boldsymbol y_n:=\{y_i^n\}_{i=1}^{k_n}$ such that
$$ y_1^n= x_1^n+ \frac{\lambda_1^n}{nk_n}, \text{ and }
y_{i}^n= x_i^n, i=2,\cdots,k_n.
$$ 
Let $v_n\in M_{\boldsymbol y_n, d}^n$. Then we have
\begin{equation}\label{eq53}
I_n(v_n)\leq I_n(u_n).
\end{equation}
In the following we will prove $I_n(v_n)>I_n(u_n)$ for large $n$, which gives a contradiction to \eqref{eq53}.

According to Proposition \ref{cor5.4}, since $\sigma(\boldsymbol{y}_n)\to 0$,
we have for large $n$,
\[
   -\Delta v_n-K_n v_n+v_n^{q-1}
=\sum_{i=1}^{k_n}\bar\lambda_i^n\cdot(x-y_i^n)(v_n)_i^\delta\quad \text{in }\R^N.  
\]
By Remark \ref{rk5.4}, we have
\begin{equation}\label{e65}
\|v_n-\sum_{i=1}^{k_n} w(\cdot-x_i^n)\|_{L^\infty(\RN)}+\|v_n-\sum_{i=1}^{k_n} w(\cdot-x_i^n)\|_{H^1_{loc}}=o_n(1).
\end{equation} 
Let $h_n, \bar h_n\in H^{-1}(A(\boldsymbol x_n,  d_n))$ be such that $\langle h_n,\varphi \rangle=\int (K_n-1)u_n\varphi$ and $\langle \bar h_n,\varphi \rangle=\int (K_n-1)v_n\varphi$ for any $\varphi\in H_0^1(A(\boldsymbol x_n,  d_n))$.  
By 
  Proposition \ref{pro.lim}, we have
  \[ 
  \|u_n - v_n\|_{\boldsymbol x_n, d} 
  \leq C (\|K_n-1\|_{p,loc}\|u_n-v_n\|_{\boldsymbol x_n, d} +\max_{1\leq j\leq k_n}|\beta_{j}(u_n)-\beta_{j}(v_n)|).
  \]
  Since $\max_{1\leq j\leq k_n}|\beta_{j}(u_n)-\beta_{j}(v_n)|= |\lambda_1^n|/(nk_n)$ and 
  $\|K_n-1\|_{p,loc}\to 0$, we know 
  \[
  \|u_n - v_n\|_{\boldsymbol x_n, d} \leq C  \frac{|\lambda_1^n|}{nk_n}.
  \]
 So by $\frac1q v_n^q-\frac1q u_n^q-u_n^{q-1}(v_n-u_n)\geq 0$, we have
\[ 
\begin{aligned}
&I_n(v_n)-I_n(u_n)\\
=&I_n'(u_n)(v_n-u_n)
 +\int\left[\frac12|\nabla v_n-\nabla u_n|^2-\frac12K_n|v_n-u_n|^2
+\frac1q v_n^q-\frac1q u_n^q-u_n^{q-1}(v_n-u_n)\right]\\
\geq &I_n'(u_n)(v_n-u_n)-\int\frac12K_n|v_n-u_n|^2\\
=&I_n'(u_n)(v_n-u_n)+O(\frac {|\lambda_1^n|^2}{n^2k_n}).
\end{aligned}
\]
By \eqref{eq55} and Lemma \ref{lem2-4},
$$
\begin{aligned}
I_n'(u_n)(v_n-u_n)= & \sum_{i=1}^{k_n}\int_{B(x_i^n,\rho)}
u_n^\delta(v_n-u_n)(\lambda_i^n\cdot (x-x_i^n)) \\
=&\sum_{i=1}^{k_n} \lambda_i^n \left(\beta_i(v_n)\|(v_n)_i^\delta\|_{L^2}^2-\beta_i(u_n)\|(u_n)_i^\delta\|_{L^2}^2\right) 
+ O( \sum_{i=1}^{k_n}\|u_n-v_n\|_{L^2(B(x_i^n,\rho))}^2)\\
=&\lambda_1^n  \beta_1(v_n)\|(v_n)_1^\delta\|_{L^2}^2 
+O(\frac {|\lambda_1^n|^2}{n^2k_n}) \\
=&  \frac{|\lambda_1^n|^2}{nk_n}\left(\|(v_n)_1^\delta\|_{L^2}^2 
+O(\frac1n)\right).
\end{aligned}$$
     By \eqref{e65}, we know $\|(v_n)_1^\delta\|_{L^2}^2=\|w^\delta\|_{L^2}+o(1)$.
     Therefore, $I_n(v_n)>I_n(u_n)$ for large $n$.
\end{proof}

\section{Solutions with infinity bumps}
In this section, we will construct a solution to equation \eqref{1.1} possessing infinitely many bumps, thereby proving Theorem \ref{them2}.
Starting from the sequence of finite-bump solutions $\{u_k\}$ constructed in previous sections, we first observe that for any $j \in \mathbb{N}$, the number of bumps of $u_k$ contained in $B(0,R_j)$ exceeds $j$ for sufficiently large $k$. This key density property allows us to employ a diagonal selection argument, extracting a subsequence whose bump centers stabilize within every ball $B(0,R_j)$. The resulting subsequence converges locally uniformly to a limit function $v$, which inherits the same multi-bump structure: $v$ possesses at least $j$ bumps in each ball $B(0,R_j)$. Standard compactness arguments combined with elliptic regularity theory ensure that $v$ solves the equation \eqref{1.1}. 

\ \ \ \ \ \ 

According to the discussions in the Section 1--5, for
$d=d_0$, there is a small constant $\alpha'>0$, such that
if $\|K-1\|_{p,loc}\leq\alpha'$, then for any $k \in \mathbb{N} \setminus \{0\}$,
any $u\in M_{\boldsymbol{x}, d}$ with $\mu_d(\boldsymbol{x})=\mu_{k,d}$, we have
\begin{equation}\label{ead}
\sigma(\boldsymbol{x})\leq \sigma_3,\ \ \ \
\|K-1\|_{p,loc}\leq\alpha_0,
\end{equation}
where $d_0, \sigma_3,\alpha_0$ are given in Proposition \ref{cor5.4}. Moreover, $u$ satisfies 
\begin{equation*} 
-\Delta u-K(x)u+ u^{q-1}
=0\ \ \text{in }\R^N,\ \ \ \
I(u)=\mu_{k,d}=\mu_d(\boldsymbol{x}).
\end{equation*}

For any $R>0$, we denote by $\tau(u,R)$ the number of points around which $u$ is emerging and that are contained in $B(0,R)$. Then we have the following lemma.
\begin{lemma}\label{lem6.1}
Let  $\boldsymbol x\in \mathcal K_k$ and
$u_k\in M_{\boldsymbol{x}, d}$ satisfy $\mu_d(\boldsymbol{x})=\mu_{k,d}$. Then  for each $j$, there are $R_j>0$ and $k_j\in\mathbb N\setminus\{0\}$ such that
$\tau(u_k,R_j)\geq j$ for all $k\geq k_j$.
\end{lemma}
\begin{proof}
We assume by contradiction that there exist $m\in\mathbb N\setminus\{0\}$ and sequences $\{R_n\}\subset \mathbb{R}^+ \setminus \{0\},\ \{k_n\} \subset \mathbb{N}$, such that $R_n \to +\infty, k_n \to +\infty$ as $n \to \infty$, but $\tau(u_{k_n}, R_n)< m$ for $u_{k_n}\in M_{\boldsymbol{x}_{k_n}, d}$ with $\boldsymbol x_{k_n}\in \mathcal K_{k_n}$.

We denote by $\boldsymbol x_{k_n}=\{x_i^n\}_{i=1}^{k_n}$ the points around which $u_{k_n}$ are emerging. 
Up to a subsequence, we assume that for $j <m$
\[
x_i^n\to x_i, \text{ for } i\leq j, \quad \min_{j<i\leq k_n}|x_i^n| \to +\infty.
\]
Then    there is  $y\in \R^N\setminus B(0,R_0+a_2(K))$ satisfying $\min_{1\leq i\leq k_n} |y-x_i^n|\geq 2R_0$ for large $n$.
 By Corollary \ref{cor413}, we have 
 \begin{gather*}
 	\mu_d(\boldsymbol{x}_{k_n})\leq\mu_d(\boldsymbol{y}_n)+\mu_d(x_{k_n}^n),\ 
 \mu_d(\boldsymbol{y}_n\cup\{y\})=\mu_d(\boldsymbol{y}_n)+\mu_d(y),\quad \text{where } \boldsymbol{y}_n= \{x_i^n\}_{i\neq k_n}\in\mathcal K_{k_n-1}.
 \end{gather*}
  By Proposition \ref{muy} and 
  $|x_{k_n}^n|\geq R_n\to+\infty$, there holds for large $n$ that
\[
 \mu_d(\boldsymbol{x}_{k_n})-\mu_d(\boldsymbol{y}_n)
  \leq \mu_d(x_{k_n}^n)= m_0+o_n(1)<\mu_d(y).
\]
This gives a contradiction that for large $n$,
\[
\mu_d(\boldsymbol{y}_n\cup\{y\})\leq \mu_{k_n,d}=\mu_d(\boldsymbol{x}_{k_n})< \mu_d(\boldsymbol{y}_n)+\mu_d(y)=\mu_d(\boldsymbol{y}_n\cup\{y\}). \qedhere
\]
\end{proof}
 
Now we are ready to find a solution with infinitely many bumps.

\begin{proof}[Proof of Theorem \ref{them2}] For $k\geq 1$,
let 
$u_k\in M_{\boldsymbol{x}, d}$ with $\mu_d(\boldsymbol{x})=\mu_{k,d}$.
According to Lemma \ref{lem6.1},
 for all $j\in\mathbb{N}$, $R_{j}\in\mathbb{R}^{+}\setminus\{0\}$ and $k_{j}\in\mathbb{N}$ exist so that $\tau(u_{k},R_{j})\geq j$, for all $k>k_{j}$. On the other hand, since the points around which any $u_{k}$ is emerging have interdistances greater or equal than $2R^*-\sigma_0$, a number $L_j\geq j$ exists so that $\tau(u_{k},R_{j})\leq  L_j$, for all $k\in\mathbb{N}$. Thus, $\lim_{j\to+\infty}R_{j}=+\infty$ and, without any loss of generality, we  assume $R_{j}< R_{j+1}$, for all $j\in\mathbb{N}$.

Therefore, we can choose, for all $j$, a subsequence $(u_{k_{n}}^{j})_{n}$ of $(u_{k})_{k}$ in such a way that
\begin{align*}
&\forall j\in\mathbb{N},\quad (u_{k_{n}}^{j+1})_{n} \text{ is a subsequence of } (u_{k_{n}}^{j})_{n}, \\
&\forall j\in\mathbb{N},\ \forall n\in\mathbb{N},\quad \tau(u_{k_{n}}^{j},R_{j})=\hat{j},\quad j\leq \hat{j}\leq M_j,
\end{align*}
and for all $j\in\mathbb{N}$, the sequences of $\hat{j}$-tuples $((x_{1}^{n})^{j},\cdots,(x_{\hat{j}}^{n})^{j})_{n}$, consisting of the $\hat{j}$ points around which $u_{k_{n}}^{j}$ is emerging and that are contained in $B(0,R_j)$, are converging as $n\to\infty$.

Consider the sequence $(v_{n})_{n}$, where
\[
v_{n}:=u_{k_{n}}^{n}.
\]
By construction, $(v_{n})_{n}$ is a subsequence of $(u^{j}_{{k}_{n}})_{n}$ for all $j$; moreover the sequences of points, around which the functions of $(v_{n})_{n}$ are emerging, are converging as $n$ goes to infinity and their limit points have interdistances greater or equal than $2R^*-\sigma_0$ and make up an unbounded numerable subset of $\mathbb{R}^{N}$.
According to \eqref{ead} and Corollary \ref{cor3.2},
we have $\|v_{n}\|_{H^1(B(0,R_j))}\leq C_j$ for any  $j\in\mathbb{N}$. Considering that $R_{j}\to+\infty$ as $j\to+\infty$ and that, for all $n$, $v_{n}$ solves \eqref{1.1}, we can infer that, up to a subsequence, $(v_{n})_{n}$ uniformly converges on every compact set of $\mathbb{R}^{N}$ and that the limit function is a solution of \eqref{1.1} which has at least $j$ emerging parts around points belonging to $B(0,R_j)$.
\end{proof}

 \vskip .1in
 \noindent{\bf Conflict of Interest.}
 On behalf of all authors, the corresponding author states that there is no conflict of interest.
 \vskip .1in
 \noindent{\bf Data Availability.}
 No datasets were generated or analysed during the current study.
 
\vskip .1in
 \noindent{\bf Acknowledgement.} 
 The  research was supported by   NSFC-12371107, NSFC-11901582.
  
\vspace{0.4cm}
\bibliographystyle{abbrv}  

\end{document}